\pdfoutput=1
\RequirePackage{ifpdf}
\ifpdf 
\documentclass[pdftex]{sigma}
\else
\documentclass{sigma}
\fi

\numberwithin{equation}{section}

\newtheorem{Theorem}{Theorem}[section]
\newtheorem{Lemma}[Theorem]{Lemma}
\newtheorem{Proposition}[Theorem]{Proposition}
 { \theoremstyle{definition}
\newtheorem{Remark}[Theorem]{Remark} }

\DeclareMathOperator{\Complex}{\mathbb{C}}

\DeclareMathOperator{\rank}{rank}

\DeclareMathOperator{\res}{res}

\DeclareMathOperator{\Ibb}{\mathbb{I}}
\DeclareMathOperator{\wgt}{wgt}
\newcommand{\rmd}{\mathrm{d}}
\newcommand{\rmi}{\mathrm{i}}
\newcommand{\Arf}{\mathrm{Arf}}
\newcommand{\Jac}{\mathrm{Jac}}

\begin{document}
\allowdisplaybreaks

\newcommand{\arXivNumber}{1709.10167}

\renewcommand{\PaperNumber}{074}

\FirstPageHeading

\ShortArticleName{On Regularization of Second Kind Integrals}

\ArticleName{On Regularization of Second Kind Integrals}

\Author{Julia BERNATSKA~$^\dag$ and Dmitry LEYKIN~$^\ddag$}

\AuthorNameForHeading{J.~Bernatska and D.~Leykin}

\Address{$^\dag$~National University of Kyiv-Mohyla Academy, 2 H.~Skovorody Str., 04655 Kyiv, Ukraine}
\EmailD{\href{mailto:bernatska.julia@ukma.edu.ua}{bernatska.julia@ukma.edu.ua}, \href{mailto:jbernatska@gmail.com}{jbernatska@gmail.com}}

\Address{$^\ddag$~44/2 Harmatna Str., apt.~32, 03067 Kyiv, Ukraine}
\EmailD{\href{mailto:dmitry.leykin@gmail.com}{dmitry.leykin@gmail.com}}

\ArticleDates{Received October 03, 2017, in final form July 02, 2018; Published online July 21, 2018}

\Abstract{We obtain expressions for second kind integrals on non-hyperelliptic $(n,s)$-curves. Such a curve possesses a Weierstrass point at infinity which is a branch point where all sheets of the curve come together. The infinity serves as the basepoint for Abel's map, and the basepoint in the definition of the second kind integrals. We define second kind differentials as having a pole at the infinity, therefore the second kind integrals need to be regularized. We propose the regularization consistent with the structure of the field of Abelian functions on Jacobian of the curve. In this connection we introduce the notion of regularization constant, a uniquely defined free term in the expansion of the second kind integral over a local parameter in the vicinity of the infinity. This is a~vector with components depending on parameters of the curve, the number of components is equal to genus of the curve. Presence of the term guarantees consistency of all relations between Abelian functions constructed with the help of the second kind integrals. We propose two methods of calculating the regularization constant, and obtain these constants for $(3,4)$, $(3,5)$, $(3,7)$, and $(4,5)$-curves. By the example of $(3,4)$-curve, we extend the proposed regularization to the case of second kind integrals with the pole at an arbitrary fixed point. Finally, we propose a scheme of obtaining addition formulas, where the second kind integrals, including the proper regularization constants, are used.}


\Classification{32A55; 32W50; 35R01; 14H40; 32A15; 14H45}

\section{Introduction}

In this paper we consider second kind integrals on algebraic curves, which are second kind $0$-forms $r(x,y;[\gamma])$ obtained from second kind differential forms $\rmd r$
\begin{gather}\label{SecKIntA}
 r(x,y;[\gamma]) = \int_{[\gamma]}^{(x,y)} \rmd r(\tilde{x},\tilde{y}),
\end{gather}
where $[\gamma]$ denotes a class of homotopically equivalent paths {from a fixed basepoint to $(x,y)$, by $(\tilde{x},\tilde{y})$ we denote a point which serves as a dummy parameter of integration}. Definitions of second kind integrals were discussed in~\cite{Hodge}, where the theory of second kind differentials was developed from the theory of stacks adapted for use in the theory of functions of several complex variables. Here we define a second kind differential form as a locally exact meromorphic form, whose all singularities are poles of order greater than one.

In what follows we deal with a particular class of curves called $(n,s)$-curves, possessing a~Weierstrass point at infinity, which is the branch point where all sheets of the curve come together. We assign the infinity to the basepoint of Abel's map on the curve. An accurate definition of the $(n,s)$-curve is given below in Section~\ref{ss:nsCurve}. Throughout the paper we consider second kind differential forms with the only singularity at the infinity, except Section~\ref{IntArbP} where the singularity is located at an arbitrary point. We call the antiderivative of a second kind $1$-form $\rmd r$ with the basepoint at the infinity a second kind integral $r$, namely,
\begin{gather}\label{SecKInt}
 r(x,y) = \int_{\infty}^{(x,y)} \rmd r(\tilde{x},\tilde{y}),
\end{gather}
and define the latter on the fundamental domain of a genus $g$ $(n,s)$-curve. Suppose the curve has a~homology basis of $2g$ cycles $\{\mathfrak{a}_i, \mathfrak{b}_i \,|\, i=1,\dots,g\}$, all disjoint except for each pair~$\mathfrak{a}_i$ and~$\mathfrak{b}_i$ intersecting at one point. The fundamental domain is the one-connected domain obtained by cutting the curve along all~$2g$ cycles of the homology basis and $g$ paths connecting the basepoint with $g$ intersection points. Evidently, the integral~\eqref{SecKInt} requires regularization at the basepoint, which is the only singularity point of the integrand function.

The idea of regularization is adopted from the definition of Weierstrass zeta-function. On the Weierstrass normal form of elliptic curve $0=f(x,y)=-y^2 + 4x^3 - g_2 x - g_3$, which is the simplest $(n,s)$-curve, the second kind integral is defined by
\begin{gather*}
 r(x,y) = \int_{\infty}^{(x,y)} \frac{\tilde{x}\, \rmd \tilde{x}}{\partial_{\tilde{y}} f}.
\end{gather*}
Here and throughout the text we denote $\partial f/\partial y$ by $\partial_y f$. Let
\begin{gather*} (x,y)=\big(\wp(u),\wp'(u)\big)= \big(u^{-2},{-}2u^{-3}\big)+O(1)\end{gather*} be a parametrization in the vicinity of infinity, namely $(x,y)\to\infty$ as $u\to 0$, where $u$ serves as a local parameter. With the standard Abel's map $\mathcal{A}(x,y) = \int_{\infty}^{(x,y)} (\partial_{\tilde{y}} f)^{-1}\rmd \tilde{x}$ the second kind integral transforms into zeta function $\zeta \big(\mathcal{A}(x,y)\big) = r(x,y)$. The textbook definition of Weierstrass zeta-function, see, e.g., \cite[Chapter~XX, Section~20.4]{Whittaker}, serves as a regularization of the second kind integral
\begin{gather*}
 \zeta(u) = \frac{1}{u} - \int_{0}^{u} \bigg(\wp(\tilde{u}) - \frac{1}{\tilde{u}^2}\bigg) \rmd \tilde{u}.
\end{gather*}

In the present paper we extend this regularization to a wider collection of curves. In this connection, a constant arises to be added to the regularized second kind integral, we call it a regularization constant. As shown in \cite[Section~193]{BakerAF}, where the notion of a regularization constant arose, these constants vanish in the case of hyperelliptic curves. This can explain why it was not revealed before. In the non-hyperelliptic case, the regularization constant does not vanish, and plays an important role when the second kind integral is used to obtain relations between Abelian functions defined on the Jacobi variety of the curve. In particular, for this purpose we employ the primitive function
\begin{gather}\label{PsiDef0}
 \psi(x,y) = \exp\left\{-\int_{\infty}^{(x,y)} r(\tilde{x},\tilde{y})^{t} \rmd u(\tilde{x},\tilde{y}) \right\}
\end{gather}
introduced in \cite{BL2005}, where $\{\rmd u(x,y),\rmd r(x,y)\}$ form a special cohomology basis, and $r(x,y)$ is the antiderivative~\eqref{SecKInt} of the $\rmd r(x,y)$. In general, the primitive function $\psi(x,y;[\gamma])$ also depends on a~path $\gamma$ from the basepoint to a point~$(x,y)$, and the second kind integral is defined by~\eqref{SecKIntA}. As above, $[\gamma]$ denotes a class of homotopically equivalent paths.

The paper is organized as follows. In Section~\ref{s:Prel} we briefly recall the notions of $(n,s)$-curve and primitive function, we explain how to construct the special cohomology basis, and specify the curves under consideration. Section~\ref{s:RegInt} is devoted to the idea of regularization of the second kind integral~\eqref{SecKInt}, also the special cohomology bases on the curves under consideration are specified.

In Section \ref{s:Reg34} we show how to produce relations between Abelian functions defined on the Jacobian of $\mathcal{V}_{(3,4)}$ with the help of equality \eqref{BLRC34}, which involves the primitive function~$\psi$. Si\-mi\-lar computations in the case of $(3,5)$-curve are shown in Appendix~\ref{A:BLRs}. We emphasize that this method gives consistent relations between Abelian functions only with a correct choice of regularization constant in the definition of $r(x,y)$. Therefore, we could say that the regularization constant is unique, and consistent with the structure of Abelian function field.

The following proposition summarizes results obtained in the paper. Each regularization constant $c_{(n,s)}$ relates to the second kind integral on $(n,s)$-curve with the special cohomology basis.
\begin{Proposition}
In some non-hyperelliptic cases regularization constants are the following
 \begin{gather*}
 c_{(3,4)}(\lambda) = \big(0, -\tfrac{1}{3}\lambda_2, {-}\tfrac{1}{6}\lambda_5 \big)^t,\\
 c_{(3,5)}(\lambda) = \big({-}\tfrac{1}{2}\lambda_1, -\tfrac{4}{15}\lambda_1^2, -\tfrac{1}{3}\lambda_4, {-}\tfrac{1}{6}\lambda_7 \big)^t,\\
 c_{(3,7)}(\lambda) = \big(0,{-}\tfrac{2}{3}\lambda_2,{-}\tfrac{2}{7}\lambda_2^2,{-}\tfrac{1}{2}\lambda_5,
 {-}\tfrac{1}{3}\lambda_8,{-}\tfrac{1}{6}\lambda_{11}\big)^t,\\
 c_{(4,5)}(\lambda) = \big(0, \tfrac{7}{45}\lambda_2, -\tfrac{3}{4}\lambda_3, \tfrac{3}{10}(\lambda_6-\lambda_3^2),
 -\tfrac{1}{2}\lambda_7, {-}\tfrac{1}{4}\lambda_{11} \big)^t.
\end{gather*}
\end{Proposition}
We do not have a general formula for $c_{(n,s)}$, however we have developed a technique of finding~$c_{(n,s)}$ for virtually any pair~$(n,s)$. Below we propose two methods of computing $c_{(n,s)}$. We demonstrate both of them by carrying out calculation for the case of $(3,4)$-curve. Other cases are much more cumbersome.

\section{Preliminaries}\label{s:Prel}

\subsection[$(n,s)$-curve]{$\boldsymbol{(n,s)}$-curve}\label{ss:nsCurve}
As mentioned above, in the paper we deal with $(n,s)$-curves \cite{BEL1999}.
They are defined for co-prime integers $n$ and $s$
as \begin{gather*}\mathcal{V}_{(n,s)}=\big\{(x,y)\in \Complex^2 \,|\, f_{(n,s)}(x,y) =0\big\},\end{gather*} where
\begin{gather}\label{nsCurve}
 f_{(n,s)}(x,y) = y^n + x^s + \sum_{i=0, j=0}^{s-2,n-2} \lambda_{ns-in- js} x^i y^j,
\end{gather}
$\lambda_{j}\in \Complex$, $\lambda_{k}=0$ whenever $k\leqslant 0$. Nondegenerate $(n,s)$-curve has genus $g = \frac{1}{2}(n - 1)(s - 1)$, {this is the maximal genus of the family of curves of the form~\eqref{nsCurve}}, see~\cite{BEL1999}. The polynomial $f_{(n,s)}$ is homogeneous with respect to Sato weight defined by $\wgt x = n$, $\wgt y=s$, thus $\wgt \lambda_j=j$. In the vicinity of $(x,y)=\infty$ we introduce a parameterization on $\mathcal{V}_{(n,s)}$ by the following formulas
\begin{gather*}
x(\xi)=-\xi^{-n}, \qquad y(\xi)=-\xi^{-s}\left((-1)^s+ \sum_{i=1}^{\infty} \mu_i(\lambda) \xi^i\right), \qquad f_{(n,s)}\big(x(\xi),y(\xi)\big)=0.
\end{gather*}
Note that $\wgt \xi = -1$, and $\mu_i(\lambda)$ are polynomials in $\lambda$ with rational coefficients, $\wgt \mu_i(\lambda) = i$.

\subsection{Basis of 1-forms}\label{ss:CoHomolBasis}
Fix the basis of holomorphic $1$-forms $\rmd u = \big(x^i y^j (\partial_y f_{(n,s)})^{-1}\rmd x\big)$, $i\geqslant 0$, $j\geqslant 0$ and $i n+j s\leqslant 2g-2$ ordered by descending Sato weight. Note that $\wgt \rmd u = ({-}w_1,\dots,{-}w_g)^t$, where $\{w_1,\dots,w_g\}$ is Weierstrass sequence obtained from $n$ and $s$ with the help of generating function $t^{w_1}+\cdots +t^{w_g}= (1-t)^{-1}- (1-t^{ns}) (1-t^n)^{-1}(1-t^s)^{-1}$. Further consider meromorphic $1$-forms with poles only at infinity. Up to globally exact forms any such $1$-form can be represented as a linear combination of $2g$ differentials $\big(x^i y^j (\partial_y f_{(n,s)})^{-1} \rmd x\big)$ with $0\leqslant i \leqslant s-2$ and $0\leqslant j \leqslant n-2$, which are holomorphic on $\mathcal{V}_{(n,s)}\backslash \infty$, and the condition $i n+j s\geqslant 2g$ singles out differentials that actually have a pole at infinity. One can form a vector $\rmd r$ of $\wgt \rmd r(\xi) = (w_1,\dots,w_g)^t$ which is subject to condition
\begin{gather}\label{DiffAjst}
\res_{\xi=0} \left(\int_0^{\xi} \rmd u(\tilde{\xi})\right) \rmd r(\xi)^t = 1_g.
\end{gather}
This condition completely determines the {principle part of $\rmd r(\xi)$.} Though the holomorphic part of $\rmd r(\xi)$ is inessential in what follows, it is {worth} to note that it can be chosen so that $\rmd r(x,y)$ and $\rmd u(x,y)$ form an associated system, see \cite[Section~138]{BakerAF} and \cite[p.~131]{EEL2000}. Below in actual calculations we use the {first $\mathcal{A}(x,y)$ and second kind integrals $\mathcal{A}^\ast(x,y)$ obtained from the chosen basis of $1$-forms}, namely,
\begin{gather*}\mathcal{A}(x,y) = \int_{\infty}^{(x,y)} \rmd u(\tilde{x},\tilde{y}),\qquad \mathcal{A}^\ast(x,y) = \int_{\infty}^{(x,y)} \rmd r(\tilde{x},\tilde{y}).\end{gather*}
So $\mathcal{A}\colon \mathcal{V} \mapsto \Jac(\mathcal{V})$ denotes the standard Abel's map. The meromorphic map $\mathcal{A}^\ast$ has a pole at infinity, and requires regularization. The regularization constants~$c_{(n,s)}$ computed in this paper relate to these particular second kind integrals.

\subsection{Primitive function}
The primitive function introduced by \eqref{PsiDef0} is employed by the both methods of computing the regularization constant. More accurately, we define it as
\begin{gather}\label{PsiDef}
 \psi(x,y) = \exp\left\{-\int_{\infty}^{(x,y)} \mathcal{A}^\ast(\tilde{x},\tilde{y})^{t} \rmd \mathcal{A}(\tilde{x},\tilde{y}) \right\}.
\end{gather}
In the definition an antiderivative $\mathcal{A}^\ast(x,y)$ of the second kind differential is supposed to contain the regularization constant when parametrized. In terms of local parameter $\xi$ the following representation of~\eqref{PsiDef} can be obtained
\begin{gather*}
\psi(\xi) = \xi^g \exp \left\{-\int_0^{\xi} \big(\big(\mathcal{A}^\ast(\xi)\big)^t \rmd \mathcal{A}(\xi) + g \xi^{-1}\rmd\xi \big)\right\},
\end{gather*}
where due to condition \eqref{DiffAjst} the integrand is a holomorphic function of~$\xi$. Evidently, the primitive function is entire with a~zero at $\xi=0$ of order~$g$, where~$g$ denotes the genus of a curve.

The primitive function $\psi(\xi)$ coincides with a certain derivative of sigma function $\sigma(u)$ at $u= \mathcal{A}(\xi)$ up to a constant factor, see Remark~\ref{R:SigmaDer}, and also with a certain modification of the prime form arisen from~\cite{N2010}. Following~\cite{BL2004}, we define the sigma function of $g$ variables, which we call a~multivariate sigma function, as a~solution of a~system of heat equations with a~Schur--Weierstrass polynomial as an initial condition. Here we apply the approach to constructing multivariate sigma functions on $(n,s)$-curves after~\cite{BL2004,BL2008}. The most complete survey of the theory of multivariate sigma is given in~\cite{BEL2012}.

\subsection{Curves under consideration}
The non-hyperelliptic plane algebraic curves under consideration are unfoldings of simple Pham singularities $x^s + y^n$, that is $(n,s)=(3,4)$ or $(3,5)$. Denote by $\mathcal{V}_{(3,4)}$ a genus $3$ curve defined by the equation $f(x,y)=0$, where the polynomial~$f$ is given by
\begin{gather}\label{C34eq}
f(x,y) = y^3 + x^4 + \lambda_2 y x^2 + \lambda_5 y x + \lambda_6 x^2 + \lambda_8 y + \lambda_9 x + \lambda_{12}.
\end{gather}
Denote by $\mathcal{V}_{(3,5)}$ a genus $4$ curve defined by the equation $f(x,y)=0$, where
$f$ is given by
\begin{gather}\label{C35eq}
 f(x,y) = y^3 + x^5 + \lambda_1 y x^3 + \lambda_4 y x^2 + \lambda_6 x^3 + \lambda_7 y x + \lambda_9 x^2 + \lambda_{10} y + \lambda_{12} x + \lambda_{15}.
\end{gather}
We also consider curves $\mathcal{V}_{(3,7)}$ and $\mathcal{V}_{(4,5)}$ of genus 6 with polynomials $f$ given by
\begin{gather*}
f(x,y) = y^3 + x^7 + \lambda_2 y x^4 + \lambda_5 y x^3 + \lambda_6 x^5 + \lambda_8 y x^2 + \lambda_9 x^4 + \lambda_{11} y x + \lambda_{12} x^3 \\
\hphantom{f(x,y) =}{} + \lambda_{14} y + \lambda_{15} x^2 + \lambda_{18} x + \lambda_{21},\\
f(x,y) = y^4 + x^5 + \lambda_2 y^2 x^2 + \lambda_3 y x^3 +\lambda_6 y^2 x+\lambda_7 y x^2 + \lambda_8 x^3 +\lambda_{10} y^2 +\lambda_{11} y x \\
\hphantom{f(x,y) =}{} + \lambda_{12} x^2 + \lambda_{15} y + \lambda_{16} x + \lambda_{20},
\end{gather*}
respectively.

\section{Regularization of second kind integral in general}\label{s:RegInt}

Now we define the regularization more accurately. Consider a second kind integral $r(x,y)$ in the vicinity of its pole. Let $\xi$ be a local coordinate, and $\rmd r(\xi)=r'(\xi)\rmd \xi$, where $r(\xi)$ is $0$-form. Then~$r(\xi)$ is decomposed into singular and regular parts:
\begin{gather*}
 r(\xi) = r_{\textup{sing}}(\xi^{-1}) + r_{\textup{reg}}(\xi) + c,\qquad r_{\textup{sing}}(0)=0,\qquad r_{\textup{reg}}(0)=0.
\end{gather*}
Essentially, this decomposition is a textbook regularization, and $c$ denotes the \emph{regularization constant}. We define the \emph{regularized second kind integral} as follows
\begin{gather}\label{IntReg}
 \mathcal{A}^\ast(\xi)= r_{\textup{sing}} (\xi)
+ c(\lambda) + \int_0^{\xi} \bigl(\rmd r(\tilde{\xi})- \rmd r_{\textup{sing}} (\tilde{\xi}) \bigr).
\end{gather}
The regularization constant $c(\lambda)$ {is a vector with $g$ components, which are} polynomials in $\lambda$ with rational coefficients, and $\wgt c(\lambda) = \wgt r(\xi) = (w_1,\dots,w_g)^t$. Note that representation~\eqref{IntReg} is essentially connected to parameterization in the vicinity of infinity, and holds true within the fundamental domain. The regularization constant should not be confused with a constant of integration. The latter vanishes in order to keep $\psi$ satisfying the functional equation~\eqref{PsiFEq}. The regularization constant appears only when the second kind integral has the pole at infinity and is parameterized in the vicinity of this pole.

If one needs to define the regularization constant, we say that the correct choice of this constant makes the primitive function $\psi$ at $\xi$ equal, up to a rational factor, to the first non-vanishing derivative of sigma-function on the Abel's image of~$\xi$. We also recall {the method} of producing relations between Abelian functions, which involves the primitive function~$\psi$. The correct choice of the regularization constant is necessary for obtaining consistent relations. In what follows, we use both properties of the primitive function $\psi$ to compute the regularization constants announced in Introduction.

In the hyperelliptic case the regularization constants vanish due to the hyperelliptic involution $\iota(x,y) = (x,-y)$, that is $ = \mathcal{A}(\iota(x,y)) = -\mathcal{A}(x,y)$, or after parameterization $\mathcal{A}(-\xi) = -\mathcal{A}(\xi)$, and the fact that zeta funtions, and also second kind integrals, are odd, for more details see \cite[Section~193]{BakerAF}. Therefore, $0=r(\xi)+r(-\xi)=2c_{(2,2g+1)}(\lambda)$,

\subsection[Basis differentials and integrals on $\mathcal{V}_{(3,4)}$]{Basis differentials and integrals on $\boldsymbol{\mathcal{V}_{(3,4)}}$}
On the curve $\mathcal{V}_{(3,4)}$ punctured at infinity we fix a basis of holomorphic differentials
\begin{subequations}\label{Diffs34}
\begin{gather}
\rmd u(x,y,\lambda) = \begin{pmatrix}
 y \\ x \\ 1
 \end{pmatrix} \frac{\rmd x}{\partial_y f},\label{Diff1s34} \\
 \rmd r(x,y,\lambda) = \begin{pmatrix}
 -x^2 \\ -2 y x \\ -5 y x^2 + \frac{2}{3} \lambda_2^2 x^2 + \frac{2}{3} \lambda_5 \lambda_2 x - \lambda_6 y
 \end{pmatrix}\frac{\rmd x}{\partial_y f}. \label{Diff2s34}
\end{gather}
\end{subequations}
Introduce a local coordinate $\xi$ in the vicinity of infinity. Then we have a parameterization
\begin{gather}\label{xyParamC34}
 x(\xi) = -\xi^{-3},\qquad y(\xi) = \xi^{-4}\big({-}1 + \tfrac{1}{3}\lambda_2 \xi^2 - \tfrac{1}{3} \lambda_5 \xi^5 + O\big(\xi^6\big)\big).
\end{gather}
In terms of the local coordinate the basis holomorphic differentials acquire the form
\begin{subequations}
 \begin{gather}\label{C34D1k}
 \rmd u(\xi) =
 \begin{pmatrix}
 -1 + \frac{1}{9} \lambda_2^2 \xi^4 + O\big(\xi^6\big)\vspace{1mm}\\
 -\xi - \frac{1}{3} \lambda_2 \xi^3 + \frac{1}{3} \lambda_5 \xi^6 + O\big(\xi^7\big)\vspace{1mm}\\
 \xi^4 + \frac{1}{3} \lambda_2 \xi^6 - \frac{1}{3} \lambda_5 \xi^{9} + O\big(\xi^{10}\big)
 \end{pmatrix} \rmd \xi,\\
\label{C34D2k}
 \rmd r(\xi) =
 \begin{pmatrix} -\xi^{-2} - \frac{1}{3} \lambda_2 + \frac{1}{3} \lambda_5\xi^3 + O\big(\xi^{4}\big) \vspace{1mm}\\
 - 2 \xi^{-3} + \frac{2}{9}\lambda_2^2 \xi + O(\xi^3)\vspace{1mm}\\
 5 \xi^{-6} + \frac{1}{9} \lambda_2^2 \xi^{-2} + O(1)\end{pmatrix} \rmd \xi.
\end{gather}
\end{subequations}
The first kind integral $\mathcal{A}(\xi)=\int_0^{\xi} \rmd u(\xi)$ is a well defined holomorphic function {of local parameter~$\xi$}. On the curve $\mathcal{V}_{(3,4)}$ the regularized second kind integral $\mathcal{A}^\ast(\xi)$, defined by~\eqref{IntReg}, has the following singular part
\begin{gather*}
 r_{\textup{sing}} (\xi) =\begin{pmatrix} \xi^{-1} \\ \xi^{-2}\\ - \xi^{-5} - \frac{1}{9}\lambda_2^2 \xi^{-1}\end{pmatrix} ,\qquad
 \rmd r_{\textup{sing}} (\xi) = r'_{\textup{sing}} (\xi) \rmd \xi = \begin{pmatrix} - \xi^{-2} \\ - 2 \xi^{-3}\\
 5 \xi^{-6} + \frac{1}{9} \lambda_2^2 \xi^{-2}\end{pmatrix} \rmd \xi.
\end{gather*}
The integral on the right hand side of \eqref{IntReg} is regular, and $\rmd \mathcal{A}^\ast(\xi)/\rmd \xi = \rmd r(\xi)/\rmd \xi$.

Basis differentials of the first and second kinds defined on $\mathcal{V}_{(3,5)}$, $\mathcal{V}_{(3,7)}$, and $\mathcal{V}_{(4,5)}$ are given in Appendix~\ref{A:Diffs}, and singular parts of the second kind integrals in terms of a local coordinate in the vicinity of infinity can be found in Appendix~\ref{A:singPart}.

\section[Regularization of second kind integral on $(3,4)$-curve]{Regularization of second kind integral on $\boldsymbol{(3,4)}$-curve}\label{s:Reg34}
\begin{Theorem}\label{T:Reg34}
In the definition of regularized second kind integral \eqref{IntReg} on $\mathcal{V}_{(3,4)}$, with the basis first and second kind differentials given by \eqref{Diffs34}, the regularization constant is equal to
 \begin{gather}\label{C34const}
 c(\lambda) = \bigg(0, -\frac{\lambda_2}{3}, {-}\frac{\lambda_5}{6} \bigg)^t.
\end{gather}
\end{Theorem}
We propose two methods of proof for this type of assertions. One of them relies on vanishing properties of sigma-functions. The other one uses bilinear Hirota type differential equations satisfied by Abelian functions.
\begin{proof}[Proof~1] Consider the function $P(\xi,u)=\sigma\bigl(\mathcal{A}(\xi)-u\bigr)$. By Riemann vanishing theorem, $P(\xi,u)$ has at most $g=3$ zeros as a function of $u$. Let $u=\mathcal{A}(\xi_1)+\mathcal{A}(\xi_2)+\mathcal{A}(\xi_3)$ then \begin{gather*}
P(\xi,u) = (\xi-\xi_1)(\xi-\xi_2)(\xi-\xi_3)\sum_{k=0} \alpha_k (\xi_1,\xi_2,\xi_3,\lambda)\xi^k,
\end{gather*}
where $\alpha_k$ are entire series in $\lambda$ with symmetric polynomials in $\xi_1$, $\xi_2$, $\xi_3$ as coefficients. In particular, $\alpha_0 (\xi_1,\xi_2,\xi_3,0) = -\big(\xi_1^2+\xi_2^2+\xi_3^2 +\xi_1 \xi_2 + \xi_1 \xi_3 + \xi_2\xi_3\big)$, and $\alpha_k(\xi_1,\xi_2,\xi_3,0)=0$ at $k>0$, further $\alpha_k(0,0,0,\lambda)=0$ at $k\geqslant 0$. According to Riemann vanishing theorem, $P(\xi,u)$ vanishes identically if $\xi_1+\xi_2+\xi_3=0$ and $\xi_1\xi_2+ \xi_1\xi_3 + \xi_2 \xi_3 = 0$, that is $u\equiv 0$, since the points $\xi_1$, $\xi_2$, $\xi_3$ correspond to zeros of the function $x-a$ on the curve $\mathcal{V}_{(3,4)}$ for some $a\in \Complex$. Otherwise, $P(\xi,u)$ has exactly $g=3$ zeros (with multiplicities).

On the other hand, inverting Abel's map we come to the determinant
\begin{gather*}
 W = \begin{vmatrix}
 y(\xi_1) & x(\xi_1) & 1 \\ y(\xi_2) & x(\xi_2) & 1\\ y(\xi_3) & x(\xi_3) & 1
 \end{vmatrix} = \frac{(\xi_1-\xi_2)(\xi_2-\xi_3)(\xi_3-\xi_1)}{\xi_1^4 \xi_2^4 \xi_3^4}
 \beta (\xi_1,\xi_2,\xi_3,\lambda).
\end{gather*}
Similarly to $\alpha_k$, series $\beta$ is entire in $\lambda$, and vanishes identically if $\xi_1+\xi_2+\xi_3=0$ and $\xi_1\xi_2+ \xi_1\xi_3 + \xi_2 \xi_3 = 0$. Thus, $\beta(\xi_1,\xi_2,\xi_3,\lambda)$ and also $W$ vanish simultaneously with $\sigma(u) = - P(0,u)= \xi_1\xi_2\xi_3 \alpha_0(\xi_1,\xi_2,\xi_3,\lambda)$.

This representation through a local parameter illustrates Riemann vanishing theorem in detail, and allows to examine the case $u\equiv 0$.
\begin{Lemma}\label{L:PsiC34}
Functions $P(\xi,u)$ and $\partial_{u_1} P(\xi,u)$ vanish identically at $u\equiv 0$. Functions $\partial^2_{u_1} P(\xi,u)$ and $\partial_{u_2} P(\xi,u)$ have a zero of multiplicity $3$ at $\xi=0$. Moreover,
\begin{gather*}
 -\partial^2_{u_1} P(\xi,u)|_{u\equiv 0} = \partial_{u_2} P(\xi,u)|_{u\equiv 0} = \psi(\xi).
\end{gather*}
\end{Lemma}
\begin{proof} In fact, the three functions $\partial^2_{u_1} P(\xi,u)|_{u\equiv 0}$, $\partial_{u_2} P(\xi,u)|_{u\equiv 0}$, and $\psi(\xi)$ considered as functions of a path $\gamma$ from infinity to the point $(x(\xi),y(\xi))$ on $\mathcal{V}_{(3,4)}$ satisfy the following functional equation
\begin{gather}
 \psi(x,y;[\gamma + \chi]) = \psi(x,y;[\gamma]) \nonumber\\
 \hphantom{\psi(x,y;[\gamma + \chi]) =}{}\times \exp \left\{ - \left(\oint_{\chi} \rmd r \right)^t
 \left(\int_{\gamma} \rmd u + \frac{1}{2}\oint_{\chi} \rmd u\right) + \pi \rmi \big(\Arf(\phi_\chi) - \Arf(\phi) \big)\right\},\label{PsiFEq}
\end{gather}
where $\rmi^2=-1$, $\chi$ is a cycle, $\Arf$ stands for Arf invariant, $\phi_\chi$ and $\phi$ are Arf functions of the cycle~$\chi$ and the main Arf function of the curve, respectively, for more details see~\cite{BL2005}. The ratio of any two of the functions $\partial^2_{u_1} P(\xi,u)|_{u\equiv 0}$, $\partial_{u_2} P(\xi,u)|_{u\equiv 0}$, and $\psi(\xi)$ is a rational function on~$\mathcal{V}_{(3,4)}$. It is straightforward to verify that such ratio has no poles and so it is a constant.
\end{proof}

On one hand, we have
\begin{gather*}
 \partial^2_{u_1} P(\xi,u) \big|_{u\equiv 0} = \partial^2_{u_1} \sigma (u) \big|_{u = \mathcal{A}(\xi)}.
\end{gather*}
From the series expansion
\begin{gather}
 \sigma\big(t u_1,t^2 u_2,t^5 u_5\big) = \left(u_5 + u_1 u_2^2 - \frac{u_1^5}{20} \right)t^5 -
 \left(\frac{\lambda_2}{6} u_2^2 u_1^3 - \frac{\lambda_2 u_1^7}{2\cdot 6\cdot 7}\right)t^7 \label{C34Sigma}\\
 \qquad{} + \left(\frac{\lambda_2^2}{12} u_1 u_2^4 + \frac{\lambda_2^2}{5!} u_2^2 u_1^5
 - \frac{\lambda_2^2 u_1^9}{4\cdot 6!} \right) t^9
 - \left(\frac{\lambda_5}{6} u_5 u_2 u_1^3 + \frac{\lambda_5}{15} u_2^5 +
 \frac{3\lambda_5}{7!}u_2 u_1^8 \right)t^{10} + O\big(t^{11}\big)\nonumber
\end{gather}
we calculate directly
\begin{gather}\label{C34Psisigma}
 -\log \partial^2_{u_1} P(\xi,u) \big|_{u\equiv 0} = 3\log \xi - \frac{\lambda_5}{12} \xi^5 + O\big(\xi^6\big).
\end{gather}

On the other hand, using notation $c(\lambda)=(c_1,c_2,c_5)^t$, from \eqref{PsiDef} we obtain
\begin{gather}
 \log \psi(\xi) = 3 \log \xi + c_1 \xi + \frac{1}{6} \big(\lambda_2+3c_2 \big) \xi^2
 + \frac{\lambda_2}{6^2} \big(\lambda_2+3 c_2 \big) \xi^4\nonumber\\
 \hphantom{\log \psi(\xi) =}{} - \frac{1}{5\cdot 6^2} \big(21\lambda_5 +4 c_1 \lambda_2^2 + 36 c_5\big)\xi^5 + O\big(\xi^6\big).\label{C34Psi}
\end{gather}
By Lemma \ref{L:PsiC34} the series \eqref{C34Psisigma} and \eqref{C34Psi} are equal. Thus, we come to a~system of \emph{linear} equations, namely
\begin{gather*}
 c_1 = 0,\qquad \lambda_2+3c_2 =0,\qquad \frac{1}{5\cdot 6^2} \big(21\lambda_5 + 4 c_1 \lambda_2^2 + 36 c_5\big)= \frac{\lambda_5}{12},
\end{gather*}
and find \eqref{C34const}. {This solution is unique.}
\end{proof}

\begin{Remark}\label{R:SigmaDer} The primitive function $\psi(\xi)$ coincides, up to a factor which is a rational number, with the first non-vanishing derivative of sigma function on the Abel's image of $\xi$. The first non-vanishing derivative of sigma function has Sato weight equal to $-g$, the corresponding differential operator has Sato weight $-\wgt \sigma (u)-g$. Recall that Sato weight of sigma function is negative and calculated by the formula $-(n^2-1)(s^2-1)/24$ for $(n,s)$-curve, see~\cite{BEL1999}.
\end{Remark}
\begin{Remark}\label{R:SigmaExpansion}
Regarding series expansion \eqref{C34Sigma} for the sigma function related to $(3,4)$-curve of the form \eqref{C34eq}, it was computed on the base of the theory of multivariate sigma functions presented in \cite{BL2004,BL2008}. In \cite[ref.~15]{EEMOP2007} the reader can find a reference to page \url{http://www.ma.hw.ac.uk/Weierstrass/}, where an expansion for the sigma function related to $(3,4)$-curve with some extra parameters is presented, the expansion is obtained by J.C.~Eilbeck. The case of cyclic $(3,4)$-curve, when $\lambda_2$, $\lambda_5$ and $\lambda_8$ vanish, is considered in~\cite{EEMOP2007}, 
where a series expansion for the cyclic $(3,4)$-sigma function is proposed. Taking into account the difference in the number of parameters and signs between equations of $(3,4)$-curves, we compared~\eqref{C34Sigma}, and also higher terms which are not presented here, with the two mentioned expansions, and found that they coincide.
\end{Remark}

\begin{proof}[Proof 2]
Let $u$ and $v$ be the Abel's map images of two non-special divisors on $\mathcal{V}_{(3,4)}$, namely,
\begin{gather}\label{uDefC34}
 u = \sum_{i=1}^g \int_{\infty}^{(x_i,y_i)} \rmd u \qquad \text{and}\qquad v = \sum_{i=1}^g \int_{\infty}^{(z_i,w_i)} \rmd u,
\end{gather}
at that $f(x_i,y_i)=0$ and $f(z_i,w_i)=0$, $i=1,2,3$, and $\sigma(u) \neq 0$, $\sigma(v) \neq 0$. By the residue theorem we have
\begin{gather}\label{ZetaDef}
 \sum_{i=1}^{g} \int^{(x_i,y_i)}_{(z_i,w_i)} \rmd r = -\res_{\xi=0}\left( \mathcal{A}^{\ast}(\xi) \frac{\rmd}{\rmd \xi}\log \frac{P(\xi,u)}{P(\xi,v)}\right).
\end{gather}
Direct calculation using \eqref{C34D1k} and \eqref{IntReg} gives
\begin{gather}\label{zetaC34}
 \sum_{i=1}^{g=3} \int_{(z_i,w_i)}^{(x_i,y_i)} \rmd r = \begin{pmatrix}
 -\zeta_1(u) \\
 -\zeta_2(u) + \wp_{1,1}(u)\\
 -\zeta_5(u) + \mathcal{P}_5(u) \end{pmatrix}
 - \begin{pmatrix}
 -\zeta_1(v) \\
 -\zeta_2(v) + \wp_{1,1}(v)\\
 -\zeta_5(v) + \mathcal{P}_5(v) \end{pmatrix},
\end{gather}
where
\begin{gather*}
\mathcal{P}_5(u) = -\tfrac{5}{12} \lambda_2 \wp _{1,2}(u) -\tfrac{5}{8} \wp_{1,2,2}(u)
-\tfrac{5}{12} \wp_{1,1,1,2}(u) -\tfrac{1}{24} \wp_{1,1,1,1,1}(u)
\\
\hphantom{\mathcal{P}_5(u)}{} = -\tfrac{1}{2} \wp_{1,1}(u)\wp_{1,1,1}(u) - \tfrac{5}{2} \wp_{1,1}(u)\wp_{1,2}(u)
-\tfrac{1}{2} \wp_{1,2,2}(u) + \tfrac{1}{6} \lambda_2 \wp_{1,1,1}(u) - \tfrac{5}{12}\lambda_5.
\end{gather*}
Here we apply the relation \eqref{wp1111RelC34}, and use the following notation
\begin{gather}
\zeta_{i}(u) = \partial_{u_i} \log \sigma(u),\qquad \wp_{i,j}(u) = -\partial_{u_i} \partial_{u_j} \log \sigma(u),\nonumber \\
\wp_{i,j,\dots,k}(u) = - \partial_{u_i} \partial_{u_j} \cdots \partial_{u_k} \log \sigma(u).\label{wpNot}
\end{gather}

\begin{Lemma}\label{L:ind4WPC34}
The following relations for Abelian functions on Jacobian of $\mathcal{V}_{(3,4)}$ hold
\begin{subequations}\label{C34WP4rel}
\begin{align}
 &\wp_{1,1,1,1}(u) = 6 \wp_{1,1}^2(u)-3\wp_{2,2}(u) - 4 \lambda_2 \wp_{1,1}(u),\label{wp1111RelC34} \\
 &\wp_{1,1,1,2}(u) = 6\wp_{1,1}(u)\wp_{1,2}(u) - \lambda_2 \wp_{1,2}(u) + \lambda_5. \label{wp1112RelC34}
\end{align}
\end{subequations}
\end{Lemma}
\begin{Remark} The complete list of relations between Abelian functions on Jacobian of genus~$3$ trigonal curve are obtained in \cite{Eilb2011}, the curve is defined by an equation slightly different from~\eqref{C34eq}.
\end{Remark}
\begin{proof}Differentiating \eqref{zetaC34} over $x_1$ we obtain
\begin{gather*}
 \frac{\rmd r(x_1,y_1,\lambda)}{\rmd x_1} = \partial_{u} \begin{pmatrix}
 -\zeta_1(u) \\
 -\zeta_2(u) + \wp_{1,1}(u)\\
 -\zeta_5(u) + \mathcal{P}_5(u) \end{pmatrix} \frac{\rmd u(x_1,y_1,\lambda)}{\rmd x_1},
\end{gather*}
which produces three relations with rational functions $\mathcal{R}_k$ of order $k$ on $\mathcal{V}_{(3,4)}$
\begin{gather*}
 \mathcal{R}_6(x_1,y_1) = 0, \qquad \mathcal{R}_7(x_1,y_1) = 0, \qquad \mathcal{R}_{10}(x_1,y_1) = 0,
\end{gather*}
where
\begin{subequations}\label{RelC34}
 \begin{gather}
 \mathcal{R}_6(x,y) = x^2 + y \wp_{1,1} + x \wp_{1,2} + \wp_{1,5},\label{R6RelC34}\\
 \mathcal{R}_7(x,y) = 2 x y + y \big(\wp_{1,2} + \wp_{1,1,1}\big) + x \big(\wp_{2,2} + \wp_{1,1,2}\big)
 + \big(\wp_{2,5} + \wp_{1,1,5}\big), \label{R7RelC34}\\
 \mathcal{R}_{10}(x,y) = 5 y x^2 - \tfrac{2}{3} \lambda_2^2 x^2 - \tfrac{2}{3} \lambda_5 \lambda_2 x + \lambda_6 y + y \big(\wp_{1,5} + \partial_{u_1} \mathcal{P}_5\big) \notag\\
\hphantom{\mathcal{R}_{10}(x,y) =}{} + x \big(\wp_{2,5} + \partial_{u_2} \mathcal{P}_5\big) + \big(\wp_{5,5} + \partial_{u_5} \mathcal{P}_5\big).\label{R10RelC34}
 \end{gather}
\end{subequations}
For brevity we omit argument $u$ of Abelian functions. Clearly, the functions \eqref{RelC34} vanish on $(x_2,y_2)$ and $(x_3,y_3)$ as well.

In order to proceed, we introduce a rational function of order $10$ {on $\mathcal{V}_{(3,4)}$}
\begin{gather*}
 \varphi_{10}(x,y) = \mathcal{R}_{10}(x,y) -
 \tfrac{5}{4} \bigl(2x - \wp_{1,2} - \wp_{1,1,1}\bigr) \mathcal{R}_{7}(x,y)\\
 \hphantom{\varphi_{10}(x,y) =}{} + \tfrac{1}{6}\bigl(15 \wp_{2,2} + 15\wp_{1,1,2} + 4\lambda_2^2\bigr) \mathcal{R}_{6}(x,y).
\end{gather*}
In fact, $\varphi_{10}(x,y)=(y,x,1)\alpha(u)$ with a certain vector function $\alpha(u)$. Thus, system $\varphi_{10}(x_1,y_1)= \varphi_{10}(x_2,y_2)= \varphi_{10}(x_3,y_3)=0$ is equivalent to a relation of the form
\begin{gather*}
 \begin{pmatrix}
 y_1 & x_1 & 1 \\ y_2 & x_2 & 1\\ y_3 & x_3 & 1
 \end{pmatrix} \alpha (u) = 0.
\end{gather*}
Suppose $(x_i,y_i)$, $i=1,2,3$, are pairwise distinct, then the determinant $\left|\begin{smallmatrix} y_1 & x_1 & 1 \\ y_2 & x_2 & 1\\ y_3 & x_3 & 1 \end{smallmatrix}\right|$ does not vanish since $\sigma(u)\neq 0$. Therefore, $\alpha(u)=0$. By breaking $\alpha (u)$ into even $\alpha (u) + \alpha (-u)$ and odd $\alpha (u) - \alpha (-u)$ parts we obtain in particular
\begin{subequations}
\begin{gather}
 \mathcal{T}_1 = \wp_{1,1,1,1,1,1} + 15\wp_{1,1,2,2} - 30\wp_{1,1,1}^2 - 24\wp_{1,5} -30 \wp_{1,2}^2\nonumber\\
 \hphantom{\mathcal{T}_1 =}{} - 60\wp_{2,2}\wp_{1,1} - 16\lambda_2^2 \wp_{1,1} - 24 \lambda_6 = 0, \label{EvenRel1C34}\\
 \mathcal{T}_2 = \wp_{1,1,1,1,1,2} + 15\wp_{1,2,2,2} - 30\wp_{1,1,1}\wp_{1,1,2} + 36\wp_{2,5}\nonumber\\
 \hphantom{\mathcal{T}_2 =}{} - 90 \wp_{2,2}\wp_{1,2} - 16\lambda_2^2 \wp_{1,2} + 16 \lambda_2 \lambda_5 = 0 , \label{EvenRel2C34}\\
 \mathcal{T}_3 = \wp_{1,1,1,1,2} - 6\wp_{1,1}\wp_{1,1,2} - 6\wp_{1,2}\wp_{1,1,1} + \lambda_2 \wp_{1,1,2} = 0. \label{OddRelC34}
\end{gather}
\end{subequations}

From \eqref{R6RelC34} we see that $\phi_{6}(x,y;u) = x^2 + \wp_{1,1}(u) y + \wp_{1,2}(u) x + \wp_{1,5}(u)$ is even in $u$
and has $2g=6$ roots $(x_i,y_i)$, $i=1,\dots,6$, on curve $\mathcal{V}_{(3,4)}$. At $(x_1,y_1)$, $(x_2,y_2)$, and $(x_3,y_3)$ the function
$\phi_{7}(x,y;u) = 2 x y + \big(\wp_{1,2}(u) + \wp_{1,1,1}(u)\big) y + \big(\wp_{2,2}(u) + \wp_{1,1,2}(u)\big) x
 + \big(\wp_{2,5}(u) + \wp_{1,1,5}(u)\big)$, cf.~\eqref{R7RelC34}, vanishes. At the same time, the function $\phi_{7'}(x,y;u) = \phi_{7}(x,y;-u) = 2 x y + \big(\wp_{1,2}(u) - \wp_{1,1,1}(u)\big) y + \big(\wp_{2,2}(u) - \wp_{1,1,2}(u)\big) x + \big(\wp_{2,5}(u) - \wp_{1,1,5}(u)\big)$ vanishes
at $(x_4,y_4)$, $(x_5,y_5)$, and $(x_6,y_6)$. Consequently, the ratio $\phi_{7}(x,y;u)\phi_{7'}(x,y;u)/\phi_{6}(x,y;u)$ has no poles on $\mathcal{V}_{(3,4)}$ except the pole of order $8$ at infinity, which means that a decomposition
\begin{gather*}
 \phi_{7}\phi_{7'} - \big(a_0 y^2 + a_1 x y + a_2 x^2 + a_4 y + a_5 x + a_8\big)\phi_{6} + b_2 f(x,y) = 0
\end{gather*}
exists. Coefficients of monomials $x^i y^j$ yield an overdetermined system of $13$ linear equations with respect to $a_0$, $a_1$, $a_2$, $a_4$, $a_5$, $a_8$, $b_2$. Its compatibility condition includes in particular
\begin{subequations}
\begin{gather}
\mathcal{T}_4 = \wp _{1,1,1}^2 - 4 \wp _{1,1}^3 + 4 \wp _{2,2} \wp _{1,1}
 - \wp _{2,1}^2 + 4 \wp _{5,1} + 4 \lambda _2 \wp _{1,1}^2 = 0, \label{wp111SqRelC34}\\
\mathcal{T}_5 = \wp _{1,1,1} \wp _{2,1,1} - 4 \wp_{2,1} \wp_{1,1}^2 + \wp_{2,1} \wp_{2,2} - 2 \wp_{5,2}
 + 2 \lambda_2 \wp _{2,1} \wp_{1,1} - 2 \lambda_5 \wp _{1,1} = 0. \label{wp111wp112RelC34}
\end{gather}
\end{subequations}

Taking a linear combination of derivatives of \eqref{EvenRel1C34}, \eqref{EvenRel2C34}, and~\eqref{wp111SqRelC34} we come to
\begin{gather*}
 \partial_{u_2} \mathcal{T}_1 - \partial_{u_1} \mathcal{T}_2 + 15 \partial_{u_2} \mathcal{T}_4 = 30 \wp_{1,1,2}
 \bigl(\wp_{1,1,1,1} - 6\wp_{1,1}^2 + 3\wp_{2,2} +4\lambda_2 \wp_{1,1}\bigr) = 0.
\end{gather*}
Since $\wp_{1,1,2}(u)$ is not identically zero, this implies \eqref{wp1111RelC34}.

Next, denote $\mathcal{T}_6 = \wp_{1,1,1,1} - 6\wp_{1,1}^2 + 3\wp_{2,2} +4\lambda_2 \wp_{1,1}$. Relation \eqref{wp1112RelC34} follows from
\begin{gather*}
 \partial_{u_2} \mathcal{T}_1 - \partial_{u_1} \mathcal{T}_2 - 30\partial_{u_1} \mathcal{T}_5 +
 20 \wp_{1,1} \bigl(\partial_{u_2}\mathcal{T}_6 - \mathcal{T}_3\bigr)\\
 \qquad{} = -60 \wp_{1,1,1} \bigl(\wp_{1,1,1,2} - 6\wp_{1,1}\wp_{1,2} + 3\lambda_2 \wp_{1,2} - \lambda_5\bigr) = 0,
\end{gather*}
by similar reason.
\end{proof}

In the case of $\mathcal{V}_{(3,4)}$ we have the equality, for more detail see \cite{BL2005},
\begin{gather}\label{BLRC34}
 \frac{\sigma\big(u-\mathcal{A}(\xi)\big)\sigma\big(u+\mathcal{A}(\xi)\big)} {\psi^2(\xi)\sigma^2(u)} = \phi_{6}\big(x(\xi),y(\xi);u\big),
\end{gather}
where \eqref{xyParamC34} is applied. Indeed, both left and right hand sides are rational functions on the curve, and vanish at $2g=6$ points which are Abel's map pre-images of $u$ and $-u$. Comparing the leading terms of expansions in the vicinity of $\xi=0$ we see that the functions are equal. The expansion gives
\begin{gather*}
 0= \frac{\sigma\big(u-\mathcal{A}(\xi)\big)\sigma\big(u+\mathcal{A}(\xi)\big)}
 {\psi^2(\xi)\sigma^2(u)} - \phi_{6}\big(x(\xi),y(\xi);u\big)= -2 c_1 \xi^{-5}
 + \left(2c_1^2 - c_2 -\frac{\lambda_2}{3}\right)\xi^{-4}
\\ \hphantom{0=}{} + c_1 (\cdots) \xi^{-3}+
 \left(c_1 \big(\cdots \big) + \frac{c_2^2}{2} + \frac{\lambda_2}{6} c_2 + c_2 \wp_{1,1}
 + \frac{1}{2}\wp_{1,1}^2 -\frac{1}{4}\wp_{2,2} -\frac{1}{12}\wp_{1,1,1,1} \right)\xi^{-2}
 \\ \hphantom{0=}{} - \left(c_1 (\cdots ) - \frac{2}{5}c_5
 -\frac{7\lambda_5}{30} - \left(c_2 + \frac{\lambda_2}{6}\right) \wp_{1,2}
 - \wp_{1,1} \wp_{1,2} + \frac{1}{6}\wp_{1,1,1,2}\right)\xi^{-1} + O(1).
\end{gather*}
Applying \eqref{C34WP4rel} we find \eqref{C34const}.
\end{proof}

\begin{Remark}The equality \eqref{BLRC34} produces bilinear Hirota type equations, which give relations between Abelian functions. Note that
only correct choice of the regularization constant guarantees consistency of relations. So the Hirota type equations carry information about the regularization constant.
\end{Remark}

\section{Applications}

\subsection{Second kind integrals with poles at an arbitrary point}\label{IntArbP}
Suppose $(z,w)$ is a fixed point on the curve $\mathcal{V}_{(n,s)}$, which is characterized by $\partial_y^i f(x,y)|_{(x,y)=(z,w)}$ $=0$, $i=0,\dots,k-1$ and $\partial_y^k f(x,y)|_{(x,y)=(z,w)}\neq 0$ with $k\in \mathbb{N}$. Then a local parameterization near $(z,w)$ is defined as follows
\begin{gather}
 \big(x(\xi),y(\xi)\big)= \big(z-A_k^{k-1} \xi^k, w - A_k \xi + O\big(\xi^2\big)\big),\nonumber\\
 A_k = \left. k! \partial_x f(x,y)
 \big(\partial^k_y f(x,y)\big)^{-1} \right|_{(x,y)=(z,w)},\label{zwPrmt}
\end{gather}
indeed, $f(x(\xi),y(\xi))=O\big(\xi^{k+1}\big)$. Points with $k=1$ are regular. When $k>1$, the point $(z,w)$ is a branch point where $k$ sheets join. On $(n,s)$-curves we have $k\leqslant n < s$.

Consider the following function defined on the curve $f(x,y)=0$
\begin{gather*}
 \mathcal{R}(x,y) = \frac{f(x,y)-f(x,w)}{(x-z)(y-w)},\qquad f(z,w)=0.
\end{gather*}
Given finite values of $(x,y)$ the function $\mathcal{R}(x,y)$ has a single simple pole at $(z,w)$:
\begin{gather}
 \mathcal{R}\big(x(\xi),y(\xi) \big)= \frac{f_z/A_k}{\xi} + O(1),\label{Rser}
\end{gather}
where $f_z=\partial_z f(z,w)$.
Suppose $w^{(1)}$ satisfies $f\big(z,w^{(1)}\big)=0$ and $w^{(1)}\neq w$. Then $\mathcal{R}(x,y)$ is regular at $(z,w^{(1)})$ and
\begin{gather*}
 \mathcal{R}(z,w^{(1)})= \frac{-f_z}{w-w^{(1)}}.
\end{gather*}

Assume $(z,w)$ is a regular point, that is the case $k=1$. With the help of $\mathcal{R}(x,y)$ we construct the function having a simple pole at $(z,w)$ and vanishing at infinity, namely,
\begin{gather*}
 \mathcal{I}_1(x,y) = \frac{1}{f_w} \left(\mathcal{R}(x,y)
 -\frac{\lambda_2}{3}(x+z) - \frac{\lambda_5}{3}
 -\frac{f_w}{3 x} + \frac{f_z}{4 y} \right) + u(x,y)^t\frac{\rmd r(z,w)}{\rmd z}\\
 \hphantom{\mathcal{I}_1(x,y) =}{} - \big(r(x,y)-c(\lambda)\big)^t\frac{\rmd u(z,w)}{\rmd z},
\end{gather*}
where $r(x,y)=\big(r_1(x,y)$, $r_2(x,y)$, $r_5(x,y)\big)$ is the second kind integral containing regularization constant $c(\lambda)$, and $u(x,y)=\big(u_1(x,y)$, $u_2(x,y)$, $u_5(x,y)\big)$ is the first kind integral. The func\-tion~$\mathcal{I}_1(x,y)$ has zero of order $2g=6$ at $\infty$.

\begin{Theorem} $\mathcal{I}_1(x,y)$ is the unique second kind integral on $\mathcal{V}_{(3,4)}$ with a simple pole at $(z,w)$, $\res_{(z,w)} \mathcal{I}_1(x,y) = 1$, and zero of order $2g=6$ at $\infty$.
\end{Theorem}
\begin{proof} Since $r_1(x,y)$, $r_2(x,y)$, $r_5(x,y)$ have the only pole at $\infty$, the fact that $\res_{(z,w)} \mathcal{I}_1(x,y) = 1$ follows directly from the expansion~\eqref{Rser} for $\mathcal{R}(x,y)$. Next, using the parameterization~\eqref{xyParamC34} we obtain $\mathcal{I}_1(\xi) = O\big(\xi^6\big)$ in the vicinity of $\infty$.
\end{proof}

Differentiating $\mathcal{I}_1(x,y)$ over the parameter $z$, we obtain an integral of the second kind with a pole of any order $\ell=2,3,\dots$ at the point $(z,w)$, actually
\begin{gather*}
 \mathcal{I}_{\ell}(x,y)=\frac{\rmd^{\ell-1} \mathcal{I}_1(x,y)}{\rmd z^{\ell-1}} = \big(\partial_z - A_1 \partial_w \big)^{\ell-1} \mathcal{I}_1(x,y),\qquad \ell>1.
\end{gather*}
Applying parameterization \eqref{zwPrmt}, an expansion in the vicinity of the pole $(z,w)$ can be found
\begin{gather*}
 \mathcal{I}_{\ell} \big(x(\xi),y(\xi)\big) = \big(\partial_z - A_1 \partial_w - \partial_\xi\big)^{\ell-1}
 \mathcal{I}_1 \big(x(\xi),y(\xi)\big) = \frac{(\ell-1)!}{\xi^\ell} + O(1).
\end{gather*}
Clearly, $\mathcal{I}_\ell(x,y)$ has zero of order $2g=6$ at infinity.

In a similar way, taking into account \eqref{zwPrmt} for $k>1$, integrals $\mathcal{I}_\ell(x,y)$ can be constructed when the pole $(z,w)$ is a branch point.

\subsection{Addition formulas}\label{ss:AddF}
Below we need polynomials $\mathcal{R}_6(x,y)$ and $\mathcal{R}_7(x,y)$ from \eqref{RelC34} with coefficients evaluated at $u$ and $v$. We will use notation $\mathcal{R}_6(x,y;u)$, $\mathcal{R}_6(x,y;v)$ and so on
\begin{gather*}
\mathcal{R}_6(x,y;u) = x^2 + y \wp_{1,1}(u) + x \wp_{1,2}(u) + \wp_{1,5}(u),\\
\mathcal{R}_7(x,y;u) = 2 x y + y \big(\wp_{1,1,1}(u) + \wp_{1,2}(u)\big) + x \big(\wp_{1,1,2}(u) + \wp_{2,2}(u)\big) + \wp_{1,1,5}(u) + \wp_{2,5}(u).
\end{gather*}
Further, introduce
\begin{gather*}
M(u) = \begin{pmatrix}
\begin{matrix}\wp_{1,1}^{-1}\big(\wp_{1,5} + \tfrac{1}{4}\big(\wp_{1,1,1}^2- \wp_{1,2}^2 \big) \\
{} + \tfrac{1}{2}\big(\wp_{1,1,2}+\wp_{2,2}\big)\big)\end{matrix} &
 \tfrac{1}{2}\big(\wp_{1,1,1} + \wp_{1,2}\big) & \wp_{1,1} \vspace{1mm}\\
\begin{matrix} \wp_{1,1}^{-1}\big(\tfrac{1}{4}\big(\wp_{1,1,1} + \wp_{1,2} \big)\big(\wp_{1,1,2}+\wp_{2,2}\big)\\
{} - \tfrac{1}{2}\big(\wp_{1,1,5} + \wp_{2,5} \big)\big)\end{matrix} &
 \tfrac{1}{2}\big(\wp_{1,1,2} + \wp_{2,2}\big) & \wp_{1,2} \vspace{1mm}\\
\begin{matrix}\wp_{1,1}^{-1}\big(\tfrac{1}{2} \wp_{1,5} \big(\wp_{1,1,2}+\wp_{2,2}\big) + \tfrac{1}{4}\big(\wp_{1,1,1} \\
{} -\wp_{1,2} \big)\big(\wp_{1,1,5} + \wp_{2,5} \big)\big)\end{matrix}
 & \tfrac{1}{2}\big(\wp_{1,1,5} + \wp_{2,5}\big) & \wp_{1,5}
 \end{pmatrix}, \\
 \pi(u) = \begin{pmatrix}
 \tfrac{3}{2}\wp_{1,1} \wp_{1,2} + \tfrac{1}{2}\wp_{1,1} \wp_{1,1,1} \vspace{1mm}\\
\wp_{1,2}^2 + \tfrac{1}{2}\wp_{1,1} \wp_{2,2}
 - \wp_{1,5} + \tfrac{1}{2}\wp_{1,1}\wp_{1,1,2} \vspace{1mm}\\
\wp_{1,2}\wp_{1,5} + \tfrac{1}{2}\wp_{1,1} \wp_{2,5}
 + \tfrac{1}{2}\wp_{1,1}\wp_{1,1,5} \end{pmatrix},
\end{gather*}
where the argument $u$ of Abelian functions $\wp_{i,j}(u)$, $\wp_{i,j,k}(u)$ is omitted for brevity.

\begin{Theorem}\label{T:R9polyC34}
 Assume the rational function of order $3g=9$ on $\mathcal{V}_{(3,4)}$
\begin{gather}\label{R9polyC34}
\mathcal{R}_9(x,y) = x^3 + \alpha_1 y^2 + \alpha_2 x y + \alpha_3 x^2 + \alpha_5 y + \alpha_6 x + \alpha_9
\end{gather}
vanishes at $2g=6$ points $(x_i,y_i)$ and $(z_i,w_i)$, $i=1,2,3$, which are equivalent under Abel's map to $u$ and $v$ in Jacobian of $\mathcal{V}_{(3,4)}$, see \eqref{uDefC34}. Then
\begin{gather}
(\alpha_1,\alpha_2,\alpha_3)^t = \big(M(u)-M(v)\big)^{-1} \big(\pi(u)-\pi(v)\big),\nonumber\\
(\alpha_5,\alpha_6,\alpha_9)^t = M(u) \big(M(u)-M(v)\big)^{-1} \big(\pi(u)-\pi(v)\big) - \pi(u).\label{aCoefsC34}
\end{gather}
\end{Theorem}

\begin{proof}Our starting point is the following relation
\begin{gather*}
 \mathcal{R}_9(x,y) - A(x,y;u) \mathcal{R}_6(x,y;u) - \tfrac{1}{2} B(x,y;u) \mathcal{R}_7(x,y;u) = (y,x,1) Q(u),
\end{gather*}
where
 \begin{gather*}
 A(x,y;u) = \alpha_3 - \wp_{1,2}(u) + x + \frac{\alpha_1}{2\wp_{1,1}(u)} \big(2 y + \wp_{1,1,1}(u) + \wp_{1,2}(u)\big),\\
 B(x,y;u) = \alpha_2-\wp_{1,1}(u) -\frac{\alpha_1}{2\wp_{1,1}(u)} \big(2x - \wp_{1,1,1}(u) + \wp_{1,2}(u)\big), \\
 Q(u) = -M(u)(\alpha_1,\alpha_2,\alpha_3)^t + (\alpha_5,\alpha_6,\alpha_9)^t + \pi(u).
\end{gather*}
By the assumption $\mathcal{R}_9(x,y)$ vanishes simultaneously with $\mathcal{R}_6(x,y;u)$ and $\mathcal{R}_7(x,y;u)$ at $(x,y)=(x_i,y_i)$, $i=1,2,3$. Thus, $Q(u)=0$. Similarly, the assumption that $\mathcal{R}_9(x,y)$, $\mathcal{R}_6(x,y;v)$ and $\mathcal{R}_7(x,y;v)$ vanish together at $(x,y)=(z_i,w_i)$, $i=1,2,3$, implies $Q(v)=0$. We obtain the following linear equations
\begin{gather*}
 -M(u)(\alpha_1,\alpha_2,\alpha_3)^t + (\alpha_5,\alpha_6,\alpha_9)^t = -\pi(u),\\
 -M(v)(\alpha_1,\alpha_2,\alpha_3)^t + (\alpha_5,\alpha_6,\alpha_9)^t = -\pi(v)
\end{gather*}
that determine $\alpha_k$. The solution is given by \eqref{aCoefsC34}.
\end{proof}

\begin{Remark} In fact, all $\alpha_k$ are symmetric with respect to $u$ and $v$. Indeed, the expression for $(\alpha_1,\alpha_2,\alpha_3)^t$ is explicitly symmetric. From \eqref{aCoefsC34} we find
\begin{gather*}
 (\alpha_5,\alpha_6,\alpha_9)^t(u,v)- (\alpha_5,\alpha_6,\alpha_9)^t(v,u ) \\
 \qquad{} = \big(M(u)-M(v)\big)\big(M(u)-M(v)\big)^{-1}\big(\pi(u)-\pi(v)\big) - \big(\pi(u) - \pi(v)\big) = 0.
\end{gather*}
\end{Remark}
\begin{Remark}\label{R:NatAddF} The function of the form \eqref{R9polyC34} described by Theorem~\ref{T:R9polyC34} has $3g=9$ zeros on~$\mathcal{V}_{(3,4)}$. Besides the sets $(x_i,y_i)$, $(z_i,w_i)$, $i=1,2,3$, $\mathcal{R}_9(x,y)$ vanishes at $(s_i,t_i)$, $i=1,2,3$. Let
\begin{gather*}
 \sum_{i=1}^3 \int_{\infty}^{(s_i,t_i)} \rmd u = w.
\end{gather*}
Then by Abel's theorem $u+v+w=0$.

Finally, the addition law on Jacobian of $\mathcal{V}_{(3,4)}$ can be written in the following natural way
 \begin{gather*}
 \rank \begin{pmatrix} \Ibb_3 & M(u) & \pi(u)\\
 \Ibb_3 & M(v) & \pi(v) \\
 \Ibb_3 & M(w) & \pi(w) \end{pmatrix} {\leqslant 6 = 2g.}
 \end{gather*}
This form is natural in the sense that it is a direct analogue of the famous elliptic relation
 \begin{gather*}
 \begin{vmatrix} 1 & \wp(u) & \wp'(u)\\
 1 & \wp(v) & \wp'(v) \\
 1 & \wp(w) & \wp'(w) \end{vmatrix} =0.
 \end{gather*}
\end{Remark}

Introduce even and odd parts of $\alpha_i(u,v)$ defined by \eqref{aCoefsC34}
\begin{gather*}
 \alpha_i^{{\rm e}}(u,v) = \tfrac{1}{2}\big(\alpha_i(u,v)+\alpha_i(-u,-v)\big),\qquad \alpha_i^{{\rm o}}(u,v) = \tfrac{1}{2}\big(\alpha_i(u,v)-\alpha_i(-u,-v)\big).
\end{gather*}

\begin{Theorem}\label{T:AddF}The following addition formulas hold on Jacobian of $\mathcal{V}_{(3,4)}$
\begin{gather}
 \zeta_1(u+v) = \zeta_1(u) + \zeta_1(v) + \alpha_1^{{\rm o}},\nonumber\\
 \zeta_2(u+v) = \zeta_2(u)+\zeta_2(v) + 2\alpha_2^{{\rm o}},\nonumber\\
 \zeta_5(u+v) = \zeta_5(u)+\zeta_5(v) + \tfrac{5}{8} \big( \wp_{1,2,2}(u) + \wp_{1,2,2}(v)\big) + \tfrac{1}{24}\big(\wp_{1,1,1,1,1}(u) + \wp_{1,1,1,1,1}(v) \big) \nonumber\\
\hphantom{\zeta_5(u+v) =}{} + \tfrac{5}{8} \partial_{u_2}\partial_{v_2} \alpha_1^{{\rm o}}
 + \tfrac{1}{24} \partial_{u_1}^2\partial_{v_1}^2 \alpha_1^{{\rm o}} + 5\alpha_5^{{\rm o}} - 5 \alpha_3^{{\rm o}} \big(\alpha_2^{{\rm e}} + (\alpha_1^{{\rm o}})^2 \big) - 5 \alpha_2^{{\rm o}} \alpha_3^{{\rm e}} \nonumber\\
\hphantom{\zeta_5(u+v) =}{} - \alpha_1^{{\rm o}} \big(5(\alpha_2^{{\rm e}})^2 + 5(\alpha_2^{{\rm o}})^2 + 5 \alpha_2^{{\rm e}} ((\alpha_1^{{\rm o}})^2-\lambda_2)
 + (\alpha_1^{{\rm o}})^4 -\tfrac{10}{3} \lambda_2 (\alpha_1^{{\rm o}})^2 + \tfrac{2}{3}\lambda_2^2\big).\label{ZetaAddF}
\end{gather}
\end{Theorem}
\begin{proof}
We need the following assertion.
\begin{Lemma}
Under the assumptions of Theorem~{\rm \ref{T:R9polyC34}} we have
 \begin{gather}\label{TriLinRF}
 \frac{\sigma(u-\mathcal{A}(\xi))\sigma(v-\mathcal{A}(\xi))\sigma(u+v+\mathcal{A}(\xi))} {\psi^3(\xi) \sigma(u)\sigma(v)\sigma(u+v)} + \mathcal{R}_9(x(\xi),y(\xi)) = 0.
 \end{gather}
\end{Lemma}
Proof follows immediately from Riemann vanishing theorem and properties of $\psi(\xi)$, for more detail see \cite{BL2005}.

Given a multi-index $\omega$ denote $\mathfrak{p}_{\omega}= \wp_{\omega}(u) + \wp_{\omega}(v) + \wp_{\omega}(w)$ {subject to} $u+v+w=0$. In the case $\# \omega =1$ we have $\mathfrak{p}_{i}= -\zeta_{i}(u) - \zeta_{i}(v) - \zeta_{i}(w)$, $i=1,2,5$. Then we aply parameteriza\-tion~\eqref{xyParamC34} and expand \eqref{TriLinRF} over $\xi$. Series coefficients vanish, that produces expressions for $\alpha_i$:
\begin{gather*}
 \alpha_1 = - \mathfrak{p}_1 + 3c_1, \notag\\
 \alpha_2 = - \tfrac{1}{2}\big(\mathfrak{p}_2-\mathfrak{p}_{1,1}+\mathfrak{p}_1^2\big) + \tfrac{1}{2}\big(\lambda_2 + 3c_2 \big) + c_1(\cdots ),\notag \\
 \alpha_3 = \tfrac{1}{6}\big(3\mathfrak{p}_{1,2}-\mathfrak{p}_{1,1,1} - 3 \mathfrak{p}_{2}\mathfrak{p}_{1} + 3\mathfrak{p}_{1,1}\mathfrak{p}_{1}-\mathfrak{p}_{1}^2\big)
 -\tfrac{1}{6} \big(\lambda_2 - 9 c_2 \big)\mathfrak{p}_1 + c_1(\cdots),
\end{gather*}
and so on. {Recall that $c_1=0$. Splitting} $\alpha_i$ into odd and even parts and applying regularization we come to expressions for $\mathfrak{p}_{1}$, $\mathfrak{p}_{2}$, $\mathfrak{p}_{5}$, where $\wp_{1,2,2}(u+v)$ and $\wp_{1,1,1,1,1}(u+v)$ are expressed in terms of derivatives of~$\alpha_1^{{\rm o}}$.
\end{proof}

\begin{Remark}In general, equalities of the form \eqref{TriLinRF} produce trilinear relations which serve as addition formulas. Note that the equality contains the primitive function $\psi(\xi)$, requiring the regularized second kind integral with the correct regularization constant. The correct regula\-ri\-za\-tion constant makes the formulas~\eqref{ZetaAddF} consistent with the `natural' addition formulas, cf. Remark~\ref{R:NatAddF}.
\end{Remark}

\section[Regularization of second kind integral on $(3,5)$-curve]{Regularization of second kind integral on $\boldsymbol{(3,5)}$-curve}\label{s:Reg35}
\begin{Theorem}\label{T:Reg35}
In the defifnition of regularized second kind integral \eqref{IntReg} on $\mathcal{V}_{(3,5)}$, {with the basis first and second kind differentials given by \eqref{Diffs35}, the regularization constant is equal to}
 \begin{gather}\label{C35const}
 c(\lambda) = \bigg({-}\frac{\lambda_1}{2},\, -\frac{4\lambda_1^2}{15},\,
 -\frac{\lambda_4}{3},\, {-}\frac{\lambda_7}{6} \bigg)^t.
\end{gather}
\end{Theorem}
We briefly consider the two methods of proof.

\begin{proof}[Proof 1] Introduce the function $P(\xi,u)=\sigma\bigl(\mathcal{A}(\xi)-u\bigr)$. According to Riemann vanishing theorem, $P(\xi,u)$ has at most $g=4$ zeros. Let $u=\mathcal{A}(\xi_1)+\mathcal{A}(\xi_2)+\mathcal{A}(\xi_3)+\mathcal{A}(\xi_4)$ then
\begin{gather*}
P(\xi,u) = (\xi-\xi_1)(\xi-\xi_2)(\xi-\xi_3)(\xi-\xi_4)\sum_{k=0}\alpha_k (\xi_1,\xi_2,\xi_3,\xi_4,\lambda)\xi^k,
\end{gather*}
where $\alpha_k$ are entire series in $\lambda$ with symmetric polynomials in $\xi_1$, $\xi_2$, $\xi_3$, $\xi_4$ as coefficients. By Riemann vanishing theorem $P(\xi)$ vanishes identically if there exists a~selection of three items from four zeros such that $\xi_1+\xi_2+\xi_3=0$ and $\xi_1\xi_2+ \xi_1\xi_3+ \xi_2 \xi_3 = 0$, that is $u\equiv 0$, since the points $\xi_1$, $\xi_2$, $\xi_3$ correspond to zeros of the function $x-a$ on the trigonal curve $\mathcal{V}_{(3,5)}$ for some $a\in \Complex$. Otherwise, $P(\xi,u)$ has exactly $g=4$ zeros (with multiplicities).

Consider the case $u\equiv 0$.
\begin{Lemma}\label{L:PsiC35} Functions $P(\xi,u)$, $\partial_{u_1}\! P(\xi,u)$, $\partial_{u_1}^2\! P(\xi,u)$, $\partial_{u_2} \!P(\xi,u)$, $\partial_{u_1}^3 \!P(\xi,u)$ and $\partial_{u_1} \partial_{u_2} P(\xi,u)$ vanish identically at $u\equiv 0$. Functions $\partial^4_{u_1} P(\xi,u)$, $\partial_{u_2}^2 P(\xi,u)$, $\partial_{u_2}\partial^2_{u_1} P(\xi,u)$, and $\partial_{u_4} P(\xi,u)$ have a zero of multiplicity $4$ at $\xi=0$. Moreover,
\begin{gather*}
 \frac{1}{3} \partial^4_{u_1} P(\xi,u)|_{u\equiv 0} = - \partial_{u_2} \partial_{u_1}^2 P(\xi,u)|_{u\equiv 0} = - \partial^2_{u_2} P(\xi,u)|_{u\equiv 0} = - \partial_{u_4} P(\xi,u)|_{u\equiv 0} = \psi(\xi).
\end{gather*}
\end{Lemma}
\begin{proof}It is straightforward to verify that the ratio of any two functions in the equality has no poles, and so is a constant.
\end{proof}

On one hand, we have
\begin{gather*}
 \partial_{u_4} P(\xi,u) \big|_{u\equiv 0} = \partial_{u_4} \sigma (u) \big|_{u = \mathcal{A}(\xi)}.
\end{gather*}
Series expansion for $\sigma_{(3,5)}$ has the form
\begin{gather}\label{C35Sigma}
 \sigma\big(t u_1,t^2 u_2,t^4 u_4,t^7 u_7\big) = S_8 t^8 + S_9 t^9 + \cdots + S_{15} t^{15} + O\big(t^{16}\big),
\end{gather}
the terms $S_k$, $k=8,9,\dots,15$ are given in Appendix~\ref{A:TermsC35}. From \eqref{C35Sigma} we calculate directly
\begin{gather}
 -\log \partial_{u_4} P(\xi,u) \big|_{u\equiv 0} = 4\log \xi - \frac{\lambda_1}{6} \xi +
 \frac{\lambda_1^2}{36}\xi^2 + \frac{1}{4}\left(\frac{\lambda_4}{3} - \frac{5 \lambda_1^4}{3^5\cdot 2}\right) \xi^4
 + \frac{7 \lambda_1^5}{3^6 \cdot 10} \xi^5 \nonumber\\
 \qquad{}- \frac{1}{60}\big(6\lambda_6 + \lambda_1^2 \lambda_4\big) \xi^6
 - \left(\frac{\lambda_7}{10} - \frac{2\lambda_6 \lambda_1}{9\cdot 5}
 - \frac{23 \lambda_4 \lambda_1^3}{3^5\cdot 10} + \frac{22 \lambda_1^7}{3^9\cdot 7}\right) \xi^7 + O\big(\xi^8\big).\label{C35Psisigma}
\end{gather}

On the other hand, using notation $c(\lambda)=(c_1,c_2,c_4,c_7)^t$, from \eqref{PsiDef} we obtain
\begin{gather}
\log \psi(\xi) = 4 \log \xi - \bigg(c_1 + \frac{2 \lambda_1}{3}\bigg) \xi
 -\frac{1}{2} \left(c_2 - \frac{\lambda_1}{3} c_1 + \frac{2\lambda_1^2}{9\cdot 5} \right) \xi^2\nonumber\\
\hphantom{\log \psi(\xi) =}{} + \frac{1}{4} \left(c_4 + \frac{\lambda_1^2}{9}c_2 - \frac{5\lambda_1^3}{3^4} c_1 -
 \frac{14 \lambda_1^4}{3^5 \cdot 5} + \frac{2\lambda_4}{3} \right) \xi^4 +
 (\cdots) \xi^5 + (\cdots) \xi^6 \nonumber\\
 \hphantom{\log \psi(\xi) =}{} -
 \frac{1}{7} \left(c_7 - \frac{5 \lambda_1^3}{3^4} c_4 + \left(\frac{2\lambda_4 \lambda_1}{9}
 - \frac{2^3 \lambda_1^5}{3^6} \right) c_2 - \left(\frac{2\lambda_6}{3} + \frac{5 \lambda_4 \lambda_1^2}{3^3}
 - \frac{44 \lambda_1^6}{3^8} \right) c_1 \right.\nonumber\\
\left. \hphantom{\log \psi(\xi) =}{}+ \frac{13}{15}\lambda_7 - \frac{29}{9\cdot 5} \lambda_6\lambda_1
 - \frac{2\cdot 73 \lambda_4 \lambda_1^3}{3^5\cdot 5} + \frac{2^3 \cdot 19 \lambda_1^7}{3^9\cdot 5} \right)\xi^7 + O\big(\xi^8\big). \label{C35Psi}
\end{gather}
By Lemma \ref{L:PsiC35} the series \eqref{C35Psisigma} and \eqref{C35Psi} are equal. Thus, we find \eqref{C35const}. {This solution is unique.}
\end{proof}

\begin{proof}[Proof 2]Let {$u$ and $v$ be the Abel's map images of two non-special divisors on $\mathcal{V}_{(3,5)}$}
\begin{gather*}
 u = \sum_{i=1}^4 \int_{\infty}^{(x_i,y_i)} \rmd u \qquad \text{and}\qquad v = \sum_{i=1}^4 \int_{\infty}^{(z_i,w_i)} \rmd u,
\end{gather*}
at that $f(x_i,y_i)=0$ and $f(z_i,w_i)=0$, $i=1,2,3,4$. Then we use the formula~\eqref{ZetaDef}, which gives
\begin{gather}\label{zetaC35}
 \sum_{i=1}^4 \int_{(z_i,w_i)}^{(x_i,y_i)} \rmd r = \begin{pmatrix}
 -\zeta_1(u) \\
 -\zeta_2(u) - \wp_{1,1}(u)\\
 -\zeta_4(u) + \mathcal{P}_4(u) \\
 -\zeta_7(u) + \mathcal{P}_7(u) \end{pmatrix}
 - \begin{pmatrix}
 -\zeta_1(v) \\
 -\zeta_2(v) - \wp_{1,1}(v)\\
 -\zeta_4(v) + \mathcal{P}_4(v) \\
 -\zeta_7(v) + \mathcal{P}_7(v) \end{pmatrix}
\end{gather}
with notation \eqref{wpNot} and
\begin{gather*}
\mathcal{P}_4(u) = - \wp _{1,1,2}(u) + \tfrac{1}{3} \lambda_1 \wp _{1,1,1}(u) + \wp_{1,1}(u)^2 + \tfrac{1}{3} \lambda_1^2 \wp _{1,1}(u) + \tfrac{1}{3}\lambda_4, \\
\mathcal{P}_7(u) = \tfrac{1}{2} \big(\wp_{2,2}(u)\wp _{1,1,1}(u) + 3\wp_{1,2}(u)\wp _{1,1,2}(u)-\wp_{1,1}(u)\wp _{1,2,2}(u) + \wp_{1,1}(u)^2\wp _{1,1,1}(u)\big) \\
\hphantom{\mathcal{P}_7(u) =}{} + \tfrac{7}{2} \big(\wp_{1,4}(u)\wp_{1,1}(u)-\wp_{1,1}(u)^2\wp_{1,2}(u)\big) + \tfrac{1}{6} \lambda_1 \big(4\wp_{1,1,4}(u) - \wp_{1,1}(u)\wp _{1,1,2}(u) \\
\hphantom{\mathcal{P}_7(u) =}{} - 5 \wp_{1,2}(u)\wp _{1,1,1}(u) + 7 \wp_{1,1}(u)^3 \big) - \tfrac{2}{15} \lambda_1^2 \wp _{1,2,2}(u)+ \tfrac{1}{5} \lambda_1^2 \wp_{1,1}(u) \wp _{1,1,1}(u) \\
\hphantom{\mathcal{P}_7(u) =}{}+ \tfrac{1}{5} \lambda_1^3 \wp_{1,1}(u)^2 + \tfrac{1}{3} \lambda_4 \wp_{1,1,1}(u) + \lambda_1 \lambda_4 \wp_{1,1}(u)+ \tfrac{7}{15}(\lambda_7-\lambda_1\lambda_6)- \tfrac{1}{90} \lambda_1^3 \lambda_4.
\end{gather*}

\begin{Lemma}\label{L:ind4WPC35} The following relations for Abelian functions on Jacobian of $\mathcal{V}_{(3,5)}$ hold
\begin{subequations}\label{C35WP4rel}
\begin{gather}
\wp_{1,1,1,1} = 6 \wp_{1,1}^2 - 3\wp_{2,2} + 2 \lambda_1 \wp_{1,2} + \lambda_1^2 \wp_{1,1} + 2\lambda_4,\label{wp1111RelC35} \\
 \wp_{1,1,1,2} = 6\wp_{1,1}\wp_{1,2} - 3 \wp_{1,4} - \lambda_1 \wp_{2,2} + \lambda_1^2 \wp_{1,2}, \label{wp1112RelC35} \\
 \wp_{1,1,2,2} = 4\wp_{1,2}^2 + 2\wp_{1,1}\wp_{2,2} + \wp_{2,4} - \lambda_1 \wp_{1,4}, \label{wp1122RelC35} \\
 \wp_{2,2,2} = \wp_{1,1,4} - 2 \big(\wp_{1,2}\wp_{1,1,1} - \wp_{1,1}\wp_{1,1,2}\big) + \lambda_1 \wp_{1,2,2}, \label{wp222RelC35}\\
 \wp_{1,2,4} = \wp_{1,1}\wp_{1,2,2} - \tfrac{1}{2} \big(\wp_{2,2}\wp_{1,1,1}+\wp_{1,2}\wp_{1,1,2}\big) + \tfrac{1}{2}\lambda_1 \big(\wp_{1,2}\wp_{1,1,1}-\wp_{1,1}\wp_{1,1,2}\big),\label{wp124RelC35}\\
 \wp_{1,1,1}^2 = 4\wp_{1,1}^3 - 4\wp_{1,1}\wp_{2,2} + \wp_{1,2}^2 + 4\wp_{2,4} + 2\lambda_1 \wp_{1,1}\wp_{1,2} +\lambda_1^2 \wp_{1,1}^2 + 4\lambda_4 \wp_{1,1} - 4\lambda_6, \label{wp111SqRelC35} \\
\wp_{1,1,1}\wp_{1,1,2} = \tfrac{2}{3}\wp_{1,1,1,4} + 4 \wp_{1,2}\wp_{1,1}^2 - 6 \wp_{1,1}\wp_{1,4}
 - \wp_{1,2}\wp_{2,2} + \lambda_1 \big(\wp_{1,2}^2 - \wp_{1,1}\wp_{2,2}\big)\notag\\
\hphantom{\wp_{1,1,1}\wp_{1,1,2} =}{} +\tfrac{2}{3} \lambda_1 \wp_{2,4} - \tfrac{2}{3} \lambda_1^2 \wp_{1,4} + \lambda_1^2 \wp_{1,1}\wp_{1,2} -\tfrac{4}{3} \lambda_7.\label{wp1114RelC35}
\end{gather}
\end{subequations}
\end{Lemma}
\begin{proof} Differentiating \eqref{zetaC35} over $x_1$, we obtain
\begin{gather*}
 \frac{\rmd r(x_1,y_1,\lambda)}{\rmd x_1} = \partial_{u} \begin{pmatrix}
 -\zeta_1(u) \\
 -\zeta_2(u) - \wp_{1,1}(u)\\
 -\zeta_4(u) + \mathcal{P}_4(u) \\
 -\zeta_7(u) + \mathcal{P}_7(u) \end{pmatrix} \frac{\rmd u(x_1,y_1,\lambda)}{\rmd x_1},
\end{gather*}
which produces $g=4$ relations with rational functions $\mathcal{R}_k$ of order $k$ on $\mathcal{V}_{(3,5)}$
\begin{gather}\label{SerRelEqsC35}
 \mathcal{R}_8(x_1,y_1) = 0, \qquad \mathcal{R}_9(x_1,y_1) = 0, \qquad \mathcal{R}_{11}(x_1,y_1) = 0, \qquad \mathcal{R}_{17}(x_1,y_1) = 0,
\end{gather}
where
\begin{subequations}\label{RelC35}
 \begin{gather}
\mathcal{R}_8(x,y) = xy + x^2 \wp_{1,1} + y \wp_{1,2} + x \wp_{1,4} + \wp_{1,7},\label{R8RelC35}\\
\mathcal{R}_9(x,y) =2 x^3 +\lambda_1 xy + x^2 \big(\wp_{1,2} - \wp_{1,1,1}\big) + y\big(\wp_{2,2} - \wp_{1,1,2}\big) \nonumber\\
\hphantom{\mathcal{R}_9(x,y) =}{} + x\big(\wp_{2,4} - \wp_{1,1,4}\big) + \big(\wp_{2,7} - \wp_{1,1,7}\big), \label{R9RelC35} \\
\mathcal{R}_{11}(x,y) =\rho_4 + x^2 \big(\wp_{1,4} + \partial_{u_1} \mathcal{P}_4\big) + y \big(\wp_{2,5} + \partial_{u_2} \mathcal{P}_4\big) \nonumber\\
\hphantom{\mathcal{R}_{11}(x,y) =}{} + x \big(\wp_{4,4} + \partial_{u_4} \mathcal{P}_4\big) + \big(\wp_{4,7} + \partial_{u_7} \mathcal{P}_4\big), \label{R11RelC35} \\
\mathcal{R}_{14}(x,y) = \rho_7 + \big(\wp_{1,7} + x^2 \partial_{u_1} \mathcal{P}_7\big) + y\big(\wp_{2,7} + \partial_{u_2} \mathcal{P}_7\big)\nonumber\\
 \hphantom{\mathcal{R}_{14}(x,y) =}{} + x\big(\wp_{4,7} + \partial_{u_4} \mathcal{P}_7\big) + \big(\wp_{7,7} + \partial_{u_7} \mathcal{P}_7\big), \label{R14RelC35}
 \end{gather}
\end{subequations}
$\rho_4$ and $\rho_7$ are defined by \eqref{Diffs35}, see Appendix~\ref{A:Diffs}. For brevity we omit argument $u$ of Abelian functions. The equations similar to~\eqref{SerRelEqsC35} take place for the points $(x_i,y_i)$, $i=2,3,4$.

Let
\begin{gather*}
 \varphi_{11}(x,y) = \mathcal{R}_{11}(x,y) + \bigl(2\wp_{1,1} + \tfrac{1}{3} \lambda_1^2 \bigr) \mathcal{R}_{9}(x,y) - \bigl(4x - 4\wp_{1,2} + 2\lambda_1\wp_{1,1} + \tfrac{1}{3} \lambda_1^3\bigr) \mathcal{R}_{8}(x,y).
\end{gather*}
At the same time, $\varphi(x,y)=\big(x^2,y,x,1\big)\alpha(u)$ with a certain vector function $\alpha(u)$. System $\varphi(x_i,y_i)=0$, $i=1,2,3,4$, is equivalent to a relation of the form
\begin{gather*}
 \begin{pmatrix}
 x_1^2 & y_1 & x_1 & 1 \\ x_2^2 & y_2 & x_2 & 1\\ x_3^2 & y_3 & x_3 & 1 \\ x_4^2 & y_4 & x_4 & 1
 \end{pmatrix} \alpha (u) = 0.
\end{gather*}
We suppose $(x_i,y_i)$, $i=1,2,3,4$, are pairwise distinct, so the determinant of the matrix does not vanish, that is $\sigma(u)\neq 0$. Thus, $\alpha(u)=0$. By splitting $\alpha (u)$ into even $\alpha (u) + \alpha (-u)$ and odd $\alpha (u) - \alpha (-u)$ parts we obtain in particular
\begin{subequations}
\begin{gather}
\mathcal{T}_1 = \wp_{1,1,1,2} - \tfrac{1}{3} \lambda_1 \wp_{1,1,1,1} - 6 \wp_{1,1}\wp_{1,2} + 3 \wp_{1,4} + 2\lambda_1 \wp_{1,1}^2 - \tfrac{1}{3}\lambda_1^2 \wp_{1,2} \nonumber\\
 \hphantom{\mathcal{T}_1 =}{} + \tfrac{1}{3}\lambda_1^3 \wp_{1,1} + \tfrac{2}{3} \lambda_1 \lambda_4 = 0, \label{EvenRel1C35}\\
\mathcal{T}_2 =\wp_{1,1,2,2} -\tfrac{1}{9} \lambda_1^2 \wp_{1,1,1,1} -4\wp_{1,2}^2
 - 2 \wp_{1,1}\wp_{2,2} -\wp_{2,4} + \lambda_1 \wp_{1,4} - \tfrac{1}{3} \lambda_1^2 \wp_{2,2}
 + \tfrac{2}{3} \lambda_1^2 \wp_{1,1}^2 \notag \\
\hphantom{\mathcal{T}_2 =}{} + \tfrac{2}{9} \lambda_1^3 \wp_{1,2} + \tfrac{1}{9} \lambda_1^4 \wp_{1,1} +\tfrac{2}{9} \lambda_1^2 \lambda_4 = 0,\label{EvenRel2C35}\\
\mathcal{T}_3 =\wp_{1,1,2,4} - \tfrac{1}{3} \lambda_1 \wp_{1,1,1,4} - 4\wp_{1,2}\wp_{1,4} - 2\wp_{1,1}\wp_{2,4} + 4\wp_{1,7} - \wp_{4,4} +2\lambda_1 \wp_{1,1} \wp_{1,4} \notag \\
\hphantom{\mathcal{T}_3 =}{} - \tfrac{1}{3} \lambda_1^2 \wp_{2,4} + \tfrac{1}{3} \lambda_1^3 \wp_{1,4} = 0, \label{EvenRel3C35}\\
\mathcal{T}_4 = \wp_{1,1,1,1,1} + 3\wp_{1,2,2} - 12\wp_{1,1}\wp_{1,1,1} -2\lambda_1 \wp_{1,1,2} - \lambda_1^2 \wp_{1,1,1} = 0, \label{OddRel1C35} \\
\mathcal{T}_5 = \wp_{1,1,1,1,2} + 3\wp_{2,2,2} - 12\wp_{1,1}\wp_{1,1,2} -2\lambda_1 \wp_{1,2,2} - \lambda_1^2 \wp_{1,1,2} = 0.\label{OddRel2C35}
\end{gather}
\end{subequations}

From \eqref{R8RelC35} we see that $\phi_{8}(x,y;u) = xy + x^2 \wp_{1,1} + y \wp_{1,2} + x \wp_{1,4} + \wp_{1,7}$ is even in $u$ and has $2g=8$ roots $(x_i,y_i)$, $i=1,\dots,8$, on curve $\mathcal{V}_{(3,5)}$. At $(x_i,y_i)$, $i=1,2,3,4$, the function $\phi_{9}(x,y;u) = 2 x^3 +\lambda_1 xy + x^2 \big(\wp_{1,2} - \wp_{1,1,1}\big) + y\big(\wp_{2,2} - \wp_{1,1,2}\big)+ x\big(\wp_{2,4} - \wp_{1,1,4}\big) + \big(\wp_{2,7} - \wp_{1,1,7}\big)$, cf.~\eqref{R9RelC35}, vanishes. At the same time, the function $\phi_{9'}(x,y;u) = \phi_{9}(x,y;-u) = 2 x^3 +\lambda_1 xy + x^2 \big(\wp_{1,2} + \wp_{1,1,1}\big) + y\big(\wp_{2,2} + \wp_{1,1,2}\big) + x\big(\wp_{2,4} + \wp_{1,1,4}\big) + \big(\wp_{2,7} + \wp_{1,1,7}\big)$ vanishes at $(x_i,y_i)$, $i=5,6,7,8$. Consequently, the ratio $\phi_{9}(x,y;u)\phi_{9'}(x,y;u)/\phi_{8}(x,y;u)$ has no poles on $\mathcal{V}_{(3,5)}$ except the pole of order $10$ at infinity, which means that a decomposition
\begin{gather*}
 \phi_{9}\phi_{9'} - \big(a_0 y^2 + a_1 x^3 + a_2 x y + a_4 x^2 + a_5 y + a_7 x + a_{10}\big)\phi_{8} + (b_0 x+ b_3) f(x,y) = 0.
\end{gather*}
exists. \looseness=-1 Coefficients of monomials $x^i y^j$ yield an overdetermined system of $17$ linear equations with respect to $a_0$, $a_1$, $a_2$, $a_4$, $a_5$, $a_7$, $a_{10}$, $b_0$, $b_3$. Its compatibility conditions include in particular
\begin{gather*}
 \mathcal{T}_6 = \wp _{1,1,1}^2 - 4 \wp _{1,1}^3 + 4 \wp _{1,1} \wp _{2,2}
 - \wp _{2,1}^2 - 4 \wp _{2,4} - 2 \lambda_1 \wp _{1,1}\wp _{1,2} - \lambda_1^2 \wp _{1,1}^2\nonumber\\
 \hphantom{\mathcal{T}_6 =}{} - 4\lambda_4 \wp _{1,1} + 4\lambda_6 = 0, \\ 
\mathcal{T}_7 = \wp _{1,1,1}\wp _{1,1,4} - \wp_{1,1}\wp _{1,1,1}\wp _{1,1,2} -6 \wp_{1,1}^2\wp_{1,4} + 4\wp_{1,1}^3 \wp_{1,2} + 2\wp_{2,2}\wp_{1,4}
 -\wp_{1,2} \wp_{2,4} \notag\\
\hphantom{\mathcal{T}_7 =}{} - \wp_{1,1}\wp_{1,2}\wp_{2,2} - \wp_{2,7} + \lambda_1 \big(\wp_{1,1}\wp_{2,4} - \wp_{1,2}\wp_{1,4}\big)
 -\lambda_1 \wp_{1,1} \big(\wp_{1,1}\wp_{2,2} - \wp_{1,2}^2\big)\notag \\
\hphantom{\mathcal{T}_7 =}{} -\lambda_1^2 \wp_{1,1} \wp_{1,4} + \lambda_1^2 \wp_{1,1}^2 \wp_{1,2} -2\lambda_4 \wp_{1,4} +2\lambda_6 \wp_{1,2} -2\lambda_7 \wp_{1,1} + 2\lambda_9 = 0.
\end{gather*}

Taking into account \eqref{OddRel1C35} and \eqref{OddRel2C35} we simplify the derivatives $\partial_{u_1} \mathcal{T}_1$ and $\partial_{u_2} \mathcal{T}_1 - \partial_{u_1} \mathcal{T}_2$, thus a system of linear equations for $\wp_{2,2,2}$ and $\wp_{1,2,4}$ is obtained. Solving the system, we come to \eqref{wp222RelC35} and \eqref{wp124RelC35}. Next, we take the derivative $\partial_{u_1} \mathcal{T}_6$ and apply \eqref{wp124RelC35}, therefore we come to
\begin{gather*}
 2\wp_{1,1,1} \big(6\wp_{1,1}^2- \wp_{1,1,1,1} - 3\wp_{2,2} + 2\lambda_1 \wp_{1,2} +\lambda_1^2 \wp_{1,1} + 2\lambda_4\big) = 0,
\end{gather*}
which gives \eqref{wp1111RelC35}. Substituting the expression for $\wp_{1,1,1,1}$ into \eqref{OddRel1C35} and \eqref{OddRel2C35}, we find~\eqref{wp1112RelC35} and~\eqref{wp1122RelC35}. Finally, we solve the system
\begin{gather*}
 \partial_{u_4}\mathcal{T}_1 - \partial_{u_1} \mathcal{T}_3 = 0,\\ \partial_{u_4}\mathcal{T}_2 - \partial_{u_2} \mathcal{T}_3 + \tfrac{1}{3} \lambda_1 \partial_{u_1} \mathcal{T}_3 = 0
\end{gather*}
for $\wp_{1,1,7}$ and $\wp_{1,2,7}$. Using the expressions for $\wp_{1,2,7}$, $\wp_{1,2,4}$, $\wp_{1,1,1,1}$, $\wp_{1,1,1,2}$ we simplify $\partial_{u_1} \mathcal{T}_7$, and subtract $\wp_{1,1,2} \mathcal{T}_6$. Then we get
\begin{gather*}
 \partial_{u_1} \mathcal{T}_7 - \wp_{1,1,2} \mathcal{T}_6 = 3\wp_{1,1,1} \big(\wp_{1,1,1}\wp_{1,1,2} - \tfrac{2}{3}\wp_{1,1,1,4} - 4 \wp_{1,2}\wp_{1,1}^2 + 6 \wp_{1,1}\wp_{1,4} + \wp_{1,2}\wp_{2,2} \\
 \hphantom{\partial_{u_1} \mathcal{T}_7 - \wp_{1,1,2} \mathcal{T}_6=}{} - \lambda_1 \big(\wp_{1,2}^2 - \wp_{1,1}\wp_{2,2}\big)
 - \tfrac{2}{3} \lambda_1 \wp_{2,4} + \tfrac{2}{3} \lambda_1^2 \wp_{1,4} - \lambda_1^2 \wp_{1,1}\wp_{1,2} + \tfrac{4}{3} \lambda_7\big).
\end{gather*}
This gives the relation \eqref{wp1114RelC35}, and finalizes the proof of Lemma~\ref{L:ind4WPC35}.
\end{proof}

In the case of $\mathcal{V}_{(3,5)}$ we have the equality
\begin{gather}\label{BLRC35}
 0 = \frac{\sigma\big(u-\mathcal{A}(\xi)\big)\sigma\big(u+\mathcal{A}(\xi)\big)}
 {\psi^2(\xi)\sigma^2(u)} - \phi_{8}\big(x(\xi),y(\xi);u\big),
\end{gather}
where the local parametrization in the vicinity of infinity is applied, for more details see \cite{BL2005}. Again the both left and right hand sides are rational functions on the curve, and vanish at $2g=8$ points which are Abel's map preimages of $u$ and $-u$. Comparing the leading terms of expansions in the vicinity of $\xi=0$ we see that the functions are equal. Applying \eqref{C35WP4rel} to the expansion of \eqref{BLRC35} given in Appendix~\ref{A:BLRs}, we find \eqref{C35const}.
\end{proof}

\section[Regularization of second kind integrals on $(3,7)$- and $(4,5)$-curves]{Regularization of second kind integrals\\ on $\boldsymbol{(3,7)}$- and $\boldsymbol{(4,5)}$-curves}\label{s:Reg37}

\begin{Theorem}\label{T:RegCg6} In the definition of regularized second kind integrals \eqref{IntReg} on $\mathcal{V}_{(3,7)}$ and~$\mathcal{V}_{(4,5)}$, {with the basis first and second kind differentials given by \eqref{Diffs37} and \eqref{Diffs45}, respectively, the regularization constants are equal to}
 \begin{gather*}
 c_{(3,7)}(\lambda) = \left(0,-\frac{2\lambda_2}{3},-\frac{2\lambda_2^2}{7},-\frac{\lambda_5}{2},
 -\frac{\lambda_8}{3},-\frac{\lambda_{11}}{6}\right)^t, \\ 
 c_{(4,5)}(\lambda) = \left(0, \frac{7\lambda_2}{45}, -\frac{3\lambda_3}{4},\frac{3}{10}\big(\lambda_6-\lambda_3^2\big), -\frac{\lambda_7}{2}, -\frac{\lambda_{11}}{4} \right)^t. 
\end{gather*}
\end{Theorem}

The first method of computing regularization constants, see Section~\ref{s:Reg34} or~\ref{s:Reg35}, requires the series expansion for sigma function. Computational complexity increases with genus, one needs $2g$ terms with coefficients of Sato weights from $0$ to $2g-1$, which are polynomials in~$\lambda$. The total number of terms needed to be computed grows exponentially. In the case of $(3,7)$- and $(4,5)$-curves we omit this method.

The second method uses relations between Abelian functions, the number of which also increases. Actually we need $2$ relations for $\mathcal{V}_{(3,4)}$, then $7$ for $\mathcal{V}_{(3,5)}$, $16$ for $\mathcal{V}_{(3,7)}$ and $18$ for $\mathcal{V}_{(4,5)}$. Auxiliary lemmas, introducing the relations in the cases of $(3,7)$- and $(4,5)$-curves, are given in Appendices \ref{A:Reg37} and \ref{A:Reg45}. To prove Theorem~\ref{T:RegCg6} we again use the equality, similar to the one for hyperelliptic case from \cite{BL2005},
\begin{gather*}
 \frac{\sigma\big(u-\mathcal{A}(\xi)\big)\sigma\big(u+\mathcal{A}(\xi)\big)} {\psi^2(\xi)\sigma^2(u)} = \phi_{2g}\big(x(\xi),y(\xi);u\big),
\end{gather*}
where the rational function $\phi_{2g}\big(x(\xi),y(\xi);u\big)$ can be found with the help of \eqref{ZetaDef} as well as the relations between Abelian functions.

\section{Concluding remarks}
In the paper we introduce regularization procedure for the second kind integral defined on a plane algebraic curve. This regularization is an extension of the standard textbook regularization, known for the elliptic case, to a wider collection of curves, namely $(n,s)$-curves. The proposed regularization is combined with parameterization in the vicinity of infinity, which is a special point on an $(n,s)$-curve where all sheets come together. We prove the importance of a~correct regularization constant in the definition of the second kind integral~\eqref{IntReg} parameterized near infinity. The regularization constant does not serve as an integration constant, it arises only under parameterization and corresponds to the case of zero integration constant. The choice of the constant is significant in connection to the definition of the primitive function~$\psi$, see~\eqref{PsiDef}, as a function of the complex parameter~$\xi$ near infinity. As conjectured in \cite{BL2005} and proven for some particular cases in the present paper, the primitive function coincides with a certain derivative of sigma-function on the Abel's image of~$\xi$, see Remark~\ref{R:SigmaDer}. This is true only for the correct choice of the regularization constant in the definition of the second kind integral.

{\sloppy Furthermore, the primitive function occurs in polylinear equalities of the forms~\eqref{BLRC34} and~\eqref{TriLinRF}, introduced in~\cite{BL2005} with regard to hyperelliptic curves. We call the equalities polylinear, in general, or bilinear and trilinear in the mentioned cases, since the left hand sides are results of the action of polylinear operators, for more details see~\cite{BL2005}. A~bilinear equality of the type~\eqref{BLRC34} allows to obtain relations between Abelian functions, as seen from Proofs~2 of Theorems~\ref{T:Reg34} and~\ref{T:Reg35}. Such equalities are used in the proof of Theorem~\ref{T:RegCg6}. In the paper we compare the relations between Abelian functions, obtained from the bilinear equality relating to a~curve under consideration, with the same relations obtained independently, and compute a~regularization constant. Conversely, when the regularization constant of the second kind integral on the curve is known, the bilinear equality produces all relations between Abelian functions. A~trilinear equality of the type~\eqref{TriLinRF} leads to addition formulas, as we show in Section~\ref{ss:AddF}. So the primitive function provides a new way to obtain relations between Abelian functions, and addition laws. And the correct regularization constant is essential for obtaining consistent relations and true addition laws.

}

\looseness=-1 In addition to these results, we suggest a new technique of obtaining relations between Abelian functions, that is to derive the relations from \eqref{ZetaDef} where the second kind integral on a divisor is computed by the residue theorem. In the paper we consider the difference of the second kind integrals on two divisors in order to cancel the regularization constant. This technique is displayed by the examples of proving Lemmas~\ref{L:ind4WPC34} and~\ref{L:ind4WPC35}, and was used to prove Lemmas~\ref{L:ind4WPC37} and~\ref{L:ind4WPC45}. The technique is applicable to any curve, and easy to use, unlike the known way using Klein's formula. The latter demands the knowledge of Klein's fundamental 2-form, which is another problem, rather difficult for non-hyperelliptic case. Recently, in~\cite{Suz2017} it was solved for a wide class of curves including $(n,s)$-curves. In the case of trigonal curves the problem was solved before in~\cite{BEL2000}.

\appendix

\section{First kind and associated second kind differentials}\label{A:Diffs}
According to the scheme proposed in \cite{BL2008} we compute the second kind differentials associated to the standard first kind differentials. The both collections are holomorphic on the corresponding curve punctured at infinity.

First and second kind differentials on curve $\mathcal{V}_{(3,5)}$
\begin{gather}
 \rmd u(x,y,\lambda) =\begin{pmatrix}
 x^2 \\ y \\x \\ 1
 \end{pmatrix} \frac{\rmd x}{\partial_y f}, \qquad
 \rmd r(x,y,\lambda) = - \begin{pmatrix}
 y x \\ 2 x^3 + \lambda_1 y x \\ \rho_4 \\ \rho_7
 \end{pmatrix}\frac{dx}{\partial_y f}, \label{Diffs35} \\
 \rho_4 = 4 y x^2 - \tfrac{2}{3} \lambda_1^2 x^3 - \tfrac{2}{3} \lambda_4 \lambda_1 x^2, \nonumber\\
 \rho_7 = 7 y x^3 + \tfrac{4}{5} \lambda_1^3 y x^2 - 2\lambda_4 \lambda_1 x^3 + \big(3\lambda_6 + \tfrac{8}{15}\lambda_4\lambda_1^2\big) y x - \big(\tfrac{2}{3}\lambda_4^2 -\tfrac{4}{5}\lambda_6\lambda_1^2
 + \tfrac{4}{3}\lambda_7\lambda_1\big)x^2 \nonumber\\
 \hphantom{\rho_7 =}{} + \big(\lambda_9+\tfrac{4}{15}\lambda_7\lambda_1^2\big)y - \big(\tfrac{2}{3}\lambda_{10}\lambda_1-\tfrac{8}{15}\lambda_9\lambda_1^2 +\tfrac{2}{3}\lambda_7\lambda_4\big)x. \notag
\end{gather}

First and second kind differentials on curve $\mathcal{V}_{(3,7)}$
\begin{gather}
 \rmd u(x,y,\lambda) =\begin{pmatrix}
 yx \\ x^3 \\ y \\ x^2 \\x \\ 1
 \end{pmatrix} \frac{\rmd x}{\partial_y f}, \qquad
 \rmd r(x,y,\lambda) = - \begin{pmatrix}
 x^4 \\ 2 y x^2 \\ \rho_4 \\ \rho_5 \\ \rho_8 \\ \rho_{11}
 \end{pmatrix}\frac{dx}{\partial_y f}, \label{Diffs37} \\
\notag \rho_4 = 4x^5 + 2\lambda_2 y x^2 + \lambda_5 yx + 2\lambda_6 x^3, \\
 \rho_5 = 5 y x^3 - \tfrac{2}{3} \lambda_2^2 x^4 + \lambda_6 yx - \tfrac{2}{3} \lambda_2 \lambda_5 x^3 - \lambda_9 y, \nonumber\\
 \rho_8 = 8 y x^4 - \tfrac{4}{3} \lambda_2^2 x^5 + 4\lambda_6 y x^2 - 2\lambda_2 \lambda_5 x^4 + 2\lambda_9 y x - \tfrac{2}{3} \big(2\lambda_2 \lambda_8\! + \lambda_5^2\big) x^3
 - \tfrac{2}{3}\big(\lambda_2\lambda_{11} \!+ \lambda_5 \lambda_8 \big) x^2, \nonumber\\
 \rho_{11} = 11 y x^5 + \big(7\lambda_6 + \tfrac{8}{7}\lambda_2^3\big) yx^3 - \tfrac{10}{3} \lambda_2 \lambda_5 x^5 + \big(5\lambda_9
 + \tfrac{6}{7} \lambda_2^2 \lambda_5\big) y x^2 \nonumber\\
\hphantom{\rho_{11} =}{} - \big(\tfrac{8}{3} \lambda_2 \lambda_8 - \tfrac{10}{7} \lambda_2^2 \lambda_6 + \tfrac{4}{3} \lambda_5^2\big) x^4
 + \big(3\lambda_{12} + \tfrac{4}{7} \lambda_2^2 \lambda_8\big) y x
 - \big(2\lambda_2 \lambda_{11} - \tfrac{8}{7} \lambda_2^2 \lambda_9 + 2\lambda_5 \lambda_8 \big) x^3\nonumber\\
 \hphantom{\rho_{11} =}{} + \big(\lambda_{15} + \tfrac{2}{7} \lambda_2^2 \lambda_{11}\big) y - \big(\tfrac{4}{3}\lambda_2 \lambda_{14} - \tfrac{6}{7} \lambda_2^2 \lambda_{12}
 + \tfrac{4}{3} \lambda_5 \lambda_{11} + \tfrac{2}{3} \lambda_8^2 \big) x^2 \nonumber \\
\hphantom{\rho_{11} =}{} + \big(\tfrac{4}{7} \lambda_2^2 \lambda_{15} - \tfrac{2}{3} \lambda_5 \lambda_{14} - \tfrac{2}{3} \lambda_8 \lambda_{11}\big) x.\nonumber
\end{gather}

First and second kind differentials on the curve $\mathcal{V}_{(4,5)}$
\begin{gather}
 \rmd u(x,y,\lambda) =\begin{pmatrix}
 y^2 \\ yx \\ x^2 \\ y \\ x \\ 1
 \end{pmatrix} \frac{\rmd x}{\partial_y f}, \qquad
 \rmd r(x,y,\lambda) = - \begin{pmatrix}
 x^3 \\ 2 y x^2 \\ 3 y^2 x + \lambda_2 x^3 \\ \rho_6 \\ \rho_7 \\ \rho_{11}
 \end{pmatrix}\frac{dx}{\partial_y f}, \label{Diffs45} \\
\notag
 \rho_6 = 6 y x^3 + 3\lambda_3 y^2 x - \lambda_2^2 y x^2 - \tfrac{1}{2} \lambda_2 \lambda_3 x^3 + \lambda_7 y^2 + \big(2\lambda_8 - \lambda_2 \lambda_6\big) y x - \tfrac{1}{2} \lambda_2 \lambda_7 x^2, \\
 \rho_7 = 7 y^2 x^2 - \tfrac{6}{5} \lambda_2^2 y^2 x - \tfrac{23}{10} \lambda_2 \lambda_3 y x^2 + \big(2\lambda_6 - \tfrac{3}{4}\lambda_3^2\big) x^3
 + \big(\lambda_8 - \tfrac{3}{5} \lambda_2 \lambda_6\big) y^2\nonumber \\
 \hphantom{\rho_7 =}{} - \big(\tfrac{1}{2} \lambda_3 \lambda_6 +\tfrac{6}{5} \lambda_2 \lambda_7\big) yx + \big(\lambda_{10}
 - \tfrac{4}{5} \lambda_2 \lambda_8 - \tfrac{3}{4}\lambda_3 \lambda_7 \big) x^2 - \big(\tfrac{1}{10}\lambda_2 \lambda_{11} + \tfrac{1}{2}\lambda_3 \lambda_{10}\big) y,\nonumber \\
 \rho_{11} = 11 y^2 x^3 -\tfrac{20}{9} \lambda_2^2 y^2 x^2 - \tfrac{13}{3} \lambda_3 \lambda_2 y x^3
 + \big(5\lambda_8 -\tfrac{49}{15} \lambda_2 \lambda_6 + \tfrac{3}{5}\lambda_2 \lambda_3^2\big) y^2 x\nonumber\\
 \hphantom{\rho_{11} =}{}
 - \big(\tfrac{59}{18}\lambda_2 \lambda_7 + \tfrac{17}{5} \lambda_3 \lambda_6 - \tfrac{9}{10}\lambda_3^3\big)y x^2 + \big(3\lambda_{10} - \tfrac{8}{9}\lambda_2 \lambda_8
 - \tfrac{9}{4}\lambda_3\lambda_7\big) x^3 \nonumber\\
 \hphantom{\rho_{11} =}{} - \big(\tfrac{20}{9}\lambda_2 \lambda_{11} + \lambda_3 \lambda_{10} + \tfrac{21}{10} \lambda_6 \lambda_7 - \tfrac{3}{5}\lambda_3^2\lambda_7\big) y x
 + \big(2\lambda_{12} - \tfrac{10}{9}\lambda_2 \lambda_{10} -\tfrac{4}{5}\lambda_6^2
 + \tfrac{3}{10}\lambda_3^2 \lambda_6 \big) y^2\nonumber\\
 \hphantom{\rho_{11} =}{} - \big(\tfrac{4}{3} \lambda_2 \lambda_{12}
 + \tfrac{3}{2}\lambda_3 \lambda_{11} + \tfrac{2}{5} \lambda_6 \lambda_8 - \tfrac{9}{10}\lambda_3^2 \lambda_8
 + \tfrac{3}{4}\lambda_7^2 \big) x^2 \nonumber\\
\hphantom{\rho_{11} =}{} - \big(\tfrac{7}{6}\lambda_2 \lambda_{15} + \tfrac{4}{5} \lambda_6 \lambda_{11} - \tfrac{3}{10} \lambda_3^2 \lambda_{11}
 + \tfrac{1}{2}\lambda_7 \lambda_{10}\big)y \nonumber\\
\hphantom{\rho_{11} =}{} - \big(\tfrac{53}{45}\lambda_2 \lambda_{16} + \tfrac{3}{4}\lambda_3 \lambda_{15} + \tfrac{3}{5}\lambda_6 \lambda_{12}
 - \tfrac{3}{5}\lambda_3^2 \lambda_{12} + \tfrac{3}{4}\lambda_7 \lambda_{11} - \lambda_8 \lambda_{10}\big) x.\nonumber
\end{gather}

\section{Singular parts of second kind integrals}\label{A:singPart}
For $(3,5)$-, $(3,7)$- and $(4,5)$-curves the principle parts of second kind integrals in Laurent series about infinity are the following
\begin{gather*}
 \mathcal{V}_{(3,5)}\colon \ r_{\textup{sing}} (\xi) = \begin{pmatrix} -\xi^{-1} \\
 -\xi^{-2} - \frac{1}{3}\lambda_1 \xi^{-1} \vspace{1mm}\\
 \xi^{-4} + \frac{1}{9}\lambda_1^2 \xi^{-2} - \frac{2}{3^4}\lambda_1^3 \xi^{-1} \\
 r_7^{\textup{sing}} \end{pmatrix}, \\ 
 \hphantom{\mathcal{V}_{(3,5)}\colon}{} \ r_7^{\textup{sing}} =
 -\xi^{-7} + \tfrac{7}{5\cdot 9} \lambda_1^2 \xi^{-5} + \tfrac{46}{5\cdot 3^4} \lambda_1^3 \xi^{-4}
 + \big(\tfrac{2}{9}\lambda_1 \lambda_4 - \tfrac{22}{5\cdot 3^6} \lambda_1^6\big) \xi^{-2} \notag \\
\hphantom{\mathcal{V}_{(3,5)}\colon \ r_7^{\textup{sing}} =}{}
 - \big(\tfrac{2}{3} \lambda_6 + \tfrac{22}{5 \cdot 3^3} \lambda_1^2 \lambda_4
 -\tfrac{71}{5\cdot 3^8} \lambda_1^6\big) \xi^{-1}; \notag\\
 \mathcal{V}_{(3,7)}\colon \ r_{\textup{sing}} (\xi) =\begin{pmatrix} \xi^{-1} \\
 \xi^{-2} \\ -\xi^{-4} + \tfrac{1}{3}\lambda_2 \xi^{-2} \vspace{1mm}\\ -\xi^{-5} - \tfrac{1}{9}\lambda_2^2 \xi^{-1}\vspace{1mm}\\
 \xi^{-8} + \tfrac{1}{9}\lambda_2^2 \xi^{-4}
 + \big(\tfrac{2}{3}\lambda_6 + \tfrac{2}{3^4}\lambda_2^3\big) \xi^{-2}
 - \tfrac{2}{9}\lambda_2 \lambda_5 \xi^{-1}
 \\ r_{11}^{\textup{sing}} \end{pmatrix},\\ 
\hphantom{\mathcal{V}_{(3,7)}\colon}{} \ r_{11}^{\textup{sing}} = -\xi^{-11} +\tfrac{11}{7\cdot 9} \lambda_2^2 \xi^{-7} -
 \big(\tfrac{2}{3} \lambda_6 + \tfrac{68}{7\cdot 3^4} \lambda_2^3 \big)\xi^{-5}
 + \tfrac{2}{9} \lambda_2 \lambda_5 \xi^{-4} \notag \\
\hphantom{\mathcal{V}_{(3,7)}\colon \ r_{11}^{\textup{sing}} =}{}
+ \big(\tfrac{2}{3}\lambda_9 + \tfrac{32}{7\cdot 3^3}\lambda_2^2 \lambda_5\big) \xi^{-2} - \big(\tfrac{2}{9}\lambda_2 \lambda_8 - \tfrac{32}{7\cdot 3^3}\lambda_2^2 \lambda_6 +
 \tfrac{1}{9}\lambda_5^2 - \tfrac{32}{7\cdot 3^6} \lambda_2^5 \big) \xi^{-1}; \notag\\
 \mathcal{V}_{(4,5)}\colon \ r_{\textup{sing}} (\xi) =\begin{pmatrix} -\xi^{-1} \\
 \xi^{-2} \\ -\xi^{-3} - \tfrac{1}{4}\lambda_2 \xi^{-1} \vspace{1mm}\\ -\xi^{-6} - \tfrac{1}{2}\lambda_3 \xi^{-3}
 - \tfrac{1}{8} \lambda_2^2 \xi^{-2} + \tfrac{1}{8} \lambda_2 \lambda_3 \xi^{-1} \vspace{1mm}\\
 r_{7}^{\textup{sing}} \\ r_{11}^{\textup{sing}}
 \end{pmatrix}, \\ 
\hphantom{\mathcal{V}_{(4,5)}\colon}{} \ r_{7}^{\textup{sing}} = \xi^{-7} - \tfrac{7}{5\cdot 4} \lambda_2 \xi^{-5} + \tfrac{29}{5\cdot 2^5} \lambda_2^2 \xi^{-3}
 - \tfrac{11}{5\cdot 8}\lambda_2 \lambda_3 \xi^{-2} - \big(\tfrac{1}{4} \lambda_6 - \tfrac{3}{2^5} \lambda_3^2
 + \tfrac{17}{5\cdot 2^7} \lambda_2^3\big) \xi^{-1}, \notag \\
\hphantom{\mathcal{V}_{(4,5)}\colon}{} \ r_{11}^{\textup{sing}} = -\xi^{-11} + \tfrac{11}{9\cdot 4} \lambda_2 \xi^{-9}
 - \tfrac{7^2}{9\cdot 2^5} \lambda_2^2 \xi^{-7} + \tfrac{19}{9\cdot 8}\lambda_2 \lambda_3 \xi^{-6}
 - \!\big(\tfrac{11}{5\cdot 4}\lambda_6 - \tfrac{33}{5\cdot 2^5} \lambda_3^2
 - \tfrac{29}{9\cdot 2^7} \lambda_2^3\big) \xi^{-5} \notag\\
\hphantom{\mathcal{V}_{(4,5)}\colon \ r_{11}^{\textup{sing}} =}{}
 - \big(\tfrac{3}{4} \lambda_8 - \tfrac{2\cdot 16}{9\cdot 5\cdot 2^4} \lambda_2 \lambda_6
 + \tfrac{757}{9\cdot 5\cdot 2^7} \lambda_2 \lambda_3^2 - \tfrac{5^3}{9\cdot 2^{11}} \lambda_2^4\big) \xi^{-3}\nonumber\\
\hphantom{\mathcal{V}_{(4,5)}\colon \ r_{11}^{\textup{sing}} =}{} - \big(\tfrac{19}{9\cdot 8} \lambda_2 \lambda_7 + \tfrac{13}{5\cdot 8} \lambda_3 \lambda_6
 - \tfrac{17}{5\cdot 2^5} \lambda_3^3 + \tfrac{19}{9 \cdot 2^6} \big) \xi^{-2} \notag \\
\hphantom{\mathcal{V}_{(4,5)}\colon \ r_{11}^{\textup{sing}} =}{}
- \big(\tfrac{1}{4}\lambda_{10} - \tfrac{11}{9\cdot 2^4}\lambda_2 \lambda_8 - \tfrac{3}{2^4} \lambda_3 \lambda_7 + \tfrac{479}{9\cdot 5\cdot 2^7} \lambda_2^2 \lambda_6 - \tfrac{659}{5\cdot 3\cdot 2^{10}}
 \lambda_2^2 \lambda_3^2 + \tfrac{163}{9\cdot 2^{13}} \lambda_2^5\big)\xi^{-1}. \notag
\end{gather*}

\section[Sigma expansion for $(3,5)$-curve]{Sigma expansion for $\boldsymbol{(3,5)}$-curve}\label{A:TermsC35}
Series expansion for sigma function related to $(3,5)$-curve of the form \eqref{C35eq} was computed on the base of the theory of multivariate sigma-functions presented in \cite{BL2004,BL2008}. The expansion has the form
\begin{gather*}
 \sigma(t u_1,t^2 u_2,t^4 u_4,t^7 u_7) = S_8 t^8 + S_9 t^9 + \cdots + S_{15} t^{15} + O\big(t^{16}\big),
\end{gather*}
where
\begin{gather*}
S_8 = u_4^2- u_7 u_1 - u_1^2 u_2 u_4 -\frac{u_2^4}{4} -\frac{u_1^4 u_2^2}{8} + \frac{u_1^8}{7\cdot 2^6} ,\\
S_9 = -\frac{\lambda_1}{2} u_4 u_2^2 u_1 - \frac{\lambda_1}{5!} u_4 u_1^5 - \frac{\lambda_1}{2\cdot 3!} u_2^3 u_1^3 + \frac{\lambda_1}{2\cdot 5!} u_2 u_1^7,\\
S_{10} =-\frac{\lambda_1^2}{4!} u_7 u_1^3 + \frac{\lambda_1^2}{8} u_4^2 u_1^2 - \frac{\lambda_1^2}{4!} u_4 u_2 u_1^4 - \frac{\lambda_1^2}{2^5} u_2^4 u_1^2 - \frac{\lambda_1^2}{5\cdot 2^6} u_2^2 u_1^6 + \frac{89 \lambda_1^2}{20\cdot 8!} u_1^{10},\\
S_{11} = -\frac{\lambda_1^3}{2^3 \cdot 3!} u_4 u_2^2 u_1^3 - \frac{3\lambda_1^3}{4\cdot 7!} u_4 u_1^7 - \frac{\lambda_1^3}{4\cdot 5!} u_2^3 u_1^5 + \frac{23 \lambda_1^3}{3\cdot 8!} u_2 u_1^9,\\
S_{12} = -\frac{2^5\lambda_4 + \lambda_1^4}{2^4\cdot 5!} u_7 u_1^5 - \frac{\lambda_4}{2} u_4^2 u_2^2 + \frac{2^4 \lambda_4 + \lambda_1^4}{2^4 \cdot 4!} u_4^2 u_1^4 + \frac{\lambda_4}{6} u_4 u_2^3 u_1^2 + \frac{2^4 \lambda_4 - \lambda_1^4}{2^4\cdot 5!} u_4 u_2 u_1^6 \\
\hphantom{S_{12} =}{} + \frac{\lambda_4}{4!} u_2^6 + \frac{3\cdot 2^4 \lambda_4 - \lambda_1^4}{2^6\cdot 4!} u_2^4 u_1^4 -
 \frac{3}{2^3\cdot 8!}\big(7 \cdot 2^3 \lambda_4 + 3 \lambda_1^4\big) u_2^2 u_1^8\\
\hphantom{S_{12} =}{} + \frac{5}{2^5\cdot 11!} \big(2 \cdot 6^3 \lambda_4 + 5 \cdot 137 \lambda_1^4\big) u_1^{12},\\
S_{13} = \frac{\lambda_1\lambda_4}{6} u_4^3 u_1 - \frac{\lambda_1\lambda_4}{2\cdot 3!} u_4^2 u_2 u_1^3
 + \frac{\lambda_1\lambda_4}{4!} u_4 u_2^4 u_1 - \frac{\lambda_1}{2^5 \cdot 5!} \big(2^4 \lambda_4 + \lambda_1^4\big) u_4 u_2^2 u_1^5\\
 \hphantom{S_{13} =}{}
 - \frac{\lambda_1}{8!} \left(13 \lambda_4 - \frac{1}{4!}\lambda_1^4\right) u_4 u_1^9 + \frac{\lambda_1 \lambda_4}{2^3 \cdot 3!} u_2^5 u_1^3
 - \frac{3 \lambda_1}{2^5\cdot 7!} \big(3\cdot 2^4 \lambda_4 + \lambda_1^4\big) u_2^3 u_1^7\\
\hphantom{S_{13} =}{} + \frac{3 \lambda_1}{2^3\cdot 10!}\big(4! \lambda_4 + 41 \lambda_1^4\big) u_2 u_1^{11},\\
 S_{14} = \frac{1}{2^6\cdot 7!} \big(2^6\cdot 4! \lambda_6 - 2^2\cdot 4! \lambda_1^2 \lambda_4
 - \lambda_1^6 \big) u_7 u_1^7 - \lambda_6 u_4^3 u_2 + \frac{1}{2^4} \big(12 \lambda_6 - \lambda_1^2 \lambda_4 \big) u_4^2 u_2^2 u_1^2 \\
\hphantom{S_{14} =}{} + \frac{1}{2^6\cdot 6!}\big( 3\cdot 2^6 \lambda_6 + 3\cdot 2^4\lambda_1^2 \lambda_4
 + \lambda_1^6 \big) u_4^2 u_1^6 + \frac{3\lambda_6}{20} u_4 u_2^5
 + \frac{1}{3!\cdot 4!}\big(3\cdot 3!\lambda_6 + \lambda_1^2 \lambda_4\big) u_4 u_2^3 u_1^4\\
 \hphantom{S_{14} =}{} - \frac{1}{2^6\cdot 7!} \big(2^4 \cdot 39 \lambda_6 - 2^4 \cdot 19 \lambda_1^2 \lambda_4 +
 \lambda_1^6 \big) u_4 u_2 u_1^8 - \frac{1}{2^3\cdot 5!} \big(12\lambda_6 - 5\lambda_1^2 \lambda_4 \big) u_2^6 u_1^2 \\
\hphantom{S_{14} =}{} - \frac{1}{2^8\cdot 6!} \big(3\cdot 2^6 \lambda_6 - 3! 4! \lambda_1^2 \lambda_4 + \lambda_1^6 \big) u_2^4 u_1^6
 + \frac{3}{2^5\cdot 10!} \big( 2\cdot 6! \lambda_6 - 33\cdot 4! \lambda_1^2 \lambda_4 - 5 \lambda_1^6 \big) u_2^2 u_1^{10} \\
 \hphantom{S_{14} =}{} + \frac{3}{2^5\cdot 14!} \big(22 \cdot 3! \cdot 6! \lambda_6 + 2\cdot 383 \cdot 5! \lambda_1^2 \lambda_4
 + 39373\lambda_1^6\big) u_1^{14},\\
 S_{15} = -\frac{\lambda_7}{2} u_7 u_4 u_2^2 - \frac{\lambda_7}{4!} u_7 u_4 u_1^4 +\frac{2\lambda_7}{4!} u_7 u_2^3 u_1^2 + \frac{3\lambda_7}{2\cdot 5!} u_7 u_2 u_1^6\\
 \hphantom{S_{15} =}{} + \frac{1}{3! \cdot 4!} \big(4! \lambda_7 - 2 \cdot 3! \lambda_1\lambda_6 + \lambda_1^3 \lambda_4 \big) u_4^3 u_1^3
 + \frac{1}{12} \big(2 \lambda_7 + \lambda_1\lambda_6 \big) u_4^2 u_2^3 u_1\\
 \hphantom{S_{15} =}{} + \frac{1}{4 \cdot 5!}\big( 4 \lambda_7 + 6\lambda_1\lambda_6 - \lambda_1^3 \lambda_4 \big) u_4^2 u_2 u_1^5
 + \frac{1}{(4!)^2} \big(4! \lambda_7 + 6^2 \lambda_1 \lambda_6 + \lambda_1^3 \lambda_4\big) u_4 u_2^4 u_1^3\\
 \hphantom{S_{15} =}{} + \frac{1}{2^7 \cdot 7!} \big(3 \cdot 2^7 \lambda_7 - 39 \cdot 2^6 \lambda_1\lambda_6 - 3 \cdot 2^4 \lambda_1^3 \lambda_4 - \lambda_1^7 \big) u_4 u_2^2 u_1^7\\
 \hphantom{S_{15} =}{} + \frac{1}{2^5 \cdot 11!} \big(3\cdot (5!)^2 \lambda_7 - 79 \cdot 2 \cdot 3! \cdot 4! \lambda_1\lambda_6
 + 40 \cdot 19\cdot 23 \lambda_1^3 \lambda_4 - 5 \lambda_1^7 \big) u_4 u_1^{11} \\
 \hphantom{S_{15} =}{} - \frac{3}{7!}\big(10\lambda_7 + 3 \lambda_1\lambda_6 \big) u_2^7 u_1 - \frac{3}{2\cdot (5!)^2} \big(20\lambda_7 + 6 \lambda_1\lambda_6 - 5 \lambda_1^3 \lambda_4 \big) u_2^5 u_1^5 \\
 \hphantom{S_{15} =}{} + \frac{1}{2^5\cdot 9!} \big(6^2 \cdot 4! \lambda_7 + 6! \lambda_1\lambda_6 - 20^2 \lambda_1^3 \lambda_4 - \lambda_1^7 \big) u_2^3 u_1^9\\
 \hphantom{S_{15} =}{} + \frac{3}{2^6\cdot 11!} \big(2\cdot 6! \lambda_7 - 21 \cdot 2^4 \lambda_1\lambda_6 + 40 \lambda_1^3 \lambda_4 + 53\lambda_1^7 \big) u_2 u_1^{13}.
\end{gather*}

\section[Expansion of bilinear equality related to $(3,5)$-curve]{Expansion of bilinear equality related to $\boldsymbol{(3,5)}$-curve}\label{A:BLRs}
The following relation holds between sigma-function $\sigma(u)$, $u = (u_1,u_2,u_4,u_7)\in \Complex^4$, primitive function $\psi$, Abel's map $\mathcal{A}$ on curve $\mathcal{V}_{(3,5)}$, and Abelian functions $\wp_{i,j}$, $\wp_{i,j,k}$, $\wp_{i,j,k,l}$ on the Jacobian of the curve
\begin{gather*}
 0= \frac{\sigma\big(u-\mathcal{A}(\xi)\big)\sigma\big(u+\mathcal{A}(\xi)\big)} {\psi^2(\xi)\sigma^2(u)} - \phi_{8}\big(x(\xi),y(\xi);u\big)\\
 \hphantom{0}{} = (2 c_1 + \lambda_1) \xi^{-7}
 + \left(c_2 + 2c_1^2 + \frac{7\lambda_1}{3} c_1 + \frac{14 \lambda_1^2}{15} \right)\xi^{-6} \\
\hphantom{0=}{}+ \left(2 c_2 c_1 + \frac{4}{3} c_1^3 + \frac{4 \lambda_1}{3}c_2 + 2 \lambda_1 c_1^2
 + \frac{2^{6} \lambda _1^2}{45} c_1 + \frac{7 \lambda _1^3}{15}
 - 2 c_1 \wp _{1,1} - \lambda _1 \wp _{1,1}\right) \xi^{-5}\\
 \hphantom{0=}{} + \left( \frac{c_2^2-c_4}{2} + 2 c_2 c_1^2 + \frac{2}{3} c_1^4 + \frac{7 \lambda_1}{3} c_2 c_1 + \frac{10 \lambda_1}{9} c_1^3 + \frac{79 \lambda_1^2}{90} c_2
 + \frac{31 \lambda_1^2}{30} c_1^2 + \frac{509 \lambda _1^3}{3^4 \cdot 10} c_1
 - \frac{\lambda_4}{6}\right.\\
\hphantom{0=}{} + \frac{2^{5} \cdot 11 \lambda_1^4}{3^4 \cdot 25} - (2 c_1 + \lambda_1) \wp_{1,2}
 - \left(c_2 + 2c_1^2 + \frac{5 \lambda_1}{3} c_1 + \frac{3 \lambda_1^2}{5} \right) \wp_{1,1}\\
\hphantom{0=}{}
 + \frac{1}{12} \big( 6 \wp_{1,1}^2 - \wp_{1,1,1,1} - 3\wp_{2,2} + 2 \lambda_1 \wp_{1,2}
 + \lambda_1^2 \wp_{1,1} + 2\lambda_4\big) \bigg) \xi^{-4} \\
\hphantom{0=}{} + \left((c_2^2 - c_4) c_1
 + \frac{4}{3} c_2 c_1^3 + \frac{4}{15} c_1^5 - \frac{8 \lambda _1}{15} c_4
 + \frac{2}{3} c_2^2 \lambda _1 + 2 \lambda_1 c_2 c_1^2 - \frac{13 \lambda_4}{15} c_1
 + \frac{4 \lambda_1}{9} c_1^4\right.\\
\hphantom{0=}{}
+ \frac{59 \lambda _1^2}{45} c_2 c_1 + \frac{7 \lambda_1^2}{15} c_1^3 + \frac{2\lambda _1^3}{5} c_2
 + \frac{19 \lambda _1^3}{45} c_1^2 + \frac{2 \cdot 79 \lambda _1^4}{3^3 \cdot 25} c_1
 - \frac{4 \lambda _1 \lambda _4}{9} + \frac{8 \lambda_1^5}{3^3\cdot 5}\\
\hphantom{0=}{}
 - \left(c_2 + 2c_1^2 + \frac{7\lambda_1}{3} c_1 + \frac{14}{15}\lambda_1^2 \right) \wp_{1,2}\\
\hphantom{0=}{}
 - \left(2 c_1 c_2 + \frac{4}{3} c_1^3 + \lambda_1 c_2 + \frac{4 \lambda_1}{3} c_1^2 + \frac{13 \lambda_1^2}{15} c_1 + \frac{4}{15}\lambda_1^3 \right) \wp_{1,1}\\
 \hphantom{0=}{}
 + \frac{1}{6} \big({-} \wp_{1,1,1,2} + 6 \wp_{1,1}\wp_{1,2} - 3 \wp_{1,4}
 - \lambda_1 \wp_{2,2} + \lambda_1^2 \wp_{1,2}\big)\\
\hphantom{0=}{}
 - \frac{1}{6} \left(c_1+\frac{\lambda_1}{3}\right) \big( 6 \wp_{1,1}^2 - \wp_{1,1,1,1} - 3\wp_{2,2} + 2 \lambda_1 \wp_{1,2}
 + \lambda_1^2 \wp_{1,1} + 2\lambda_4\big) \bigg) \xi^{-3}\\
\hphantom{0=}{} +
 \left( \big(c_2^2 - c_4\big) c_1^2 - \frac{c_2 c_4}{2}
 + \frac{c_2^3}{6} + \frac{2}{3} c_2 c_1^4 + \frac{4}{45} c_1^6 - \frac{ \lambda _4}{2}c_2 -\frac{11\lambda_4}{15} c_1^2 - \frac{7^2}{90} \lambda _1 \lambda_4 c_1-\frac{9\lambda_1}{10} c_4 c_1\right. \\
\hphantom{0=}{}
+ \frac{7\lambda_1}{6} c_2^2 c_1 + \frac{10}{9} \lambda_1 c_2 c_1^3 + \frac{2 \lambda_1}{15} c_1^5 - \frac{13\lambda_1^2}{45} c_4 + \frac{37\lambda _1^2}{90} c_2^2 + \frac{83 \lambda_1^2}{90} c_2 c_1^2 + \frac{19 \lambda_1^2}{3^3 \cdot 5} c_1^4\\
\hphantom{0=}{} + \frac{2\cdot 109\lambda_1^3}{3^4 \cdot 5} c_2 c_1+ \frac{7\cdot 13}{3^5 \cdot 2} \lambda_1^3 c_1^3 + \frac{2 \cdot 463 \lambda_1^4}{3^5 \cdot 25} c_2 + \frac{7\cdot 89 \lambda_1^4}{3^4\cdot 50}c_1^2
 + \frac{2^2 \cdot 331 \lambda_1^5}{3^6 \cdot 25} c_1 - \frac{28}{3^3 \cdot 5} \lambda_1^2 \lambda_4 \\
\hphantom{0=}{}
 + \frac{2 \cdot 263}{3^5 \cdot 125}\lambda _1^6 + (2c_1+\lambda_1)\wp_{1,4} - \left(2 c_2 c_1 + \frac{4}{3} c_1^3 + \frac{4 \lambda_1}{3} c_2
 + 2\lambda_1 c_1^2 + \frac{2^{6} \lambda _1^2}{45} c_1 + \frac{7\lambda _1^3}{15}\right)\wp_{1,2}\\
\hphantom{0=}{}
 - \left(\frac{c_2^2-c_4}{2} + 2 c_2c_1^2 + \frac{2}{3} c_1^4 + \frac{5 \lambda_1}{3} c_2 c_1 + \frac{2 \lambda_1}{3} c_1^3 + \frac{7^2 \lambda_1^2}{90} c_2 + \frac{53 \lambda_1^2}{90} c_1^2
 + \frac{59 \lambda _1^3}{3^4\cdot 2} c_1 -\frac{\lambda _4}{6}\right.\\
 \left.\hphantom{0=}{} + \frac{197 \lambda_1^4}{3^4\cdot 25}\right) \wp_{1,1} + \frac{4}{5!} \big(4\wp_{1,1}^2 - \wp_{1,1,1}^2 + \wp_{1,2}^2
 - 4 \wp_{1,1} \wp_{2,2} + 4 \wp_{2,4} + 2 \lambda_1 \wp_{1,1} \wp_{1,2}
 + \lambda_1^2 \wp_{1,1}^2\\
\hphantom{0=}{}
+ 4 \lambda_4 \wp_{1,1} - 4\lambda_6 \big) + \frac{14}{5!}\big(4\wp_{1,2}^2 2\wp_{1,1}\wp_{2,2} + -\wp_{1,1,2,2} + \wp_{2,4} - \lambda_1 \wp_{1,4}\big) \\
\hphantom{0=}{}
 + \frac{8}{6!} (30 c_1 + 13\lambda_1) \big(6 \wp_{1,1}\wp_{1,2} - \wp_{1,1,1,2} - 3 \wp_{1,4}
 - \lambda_1 \wp_{2,2} + \lambda_1^2 \wp_{1,2} \big)\\
\hphantom{0=}{}
 - \frac{2}{6!} \big(30\big(c_2+2c_1^2 + \lambda_1 c_1\big) + \frac{22 \lambda_1^2}{3} - 18 \wp_{1,1} + \partial_{u_1}^2\big)
 \big(6 \wp_{1,1}^2 - \wp_{1,1,1,1} - 3\wp_{2,2} + 2 \lambda_1 \wp_{1,2}\\
\hphantom{0=}{} + \lambda_1^2 \wp_{1,1} + 2\lambda_4\big) \big) \xi^{-2} +
 \left( \frac{2}{7} c_7 - c_1 c_2 c_4
 + \frac{2}{3} c_1^3 \big(c_2^2- c_4\big) + \frac{1}{3}c_1 c_2^3 + \frac{4}{15} c_1^5 c_2 + \frac{8}{3^2\cdot 35} c_1^7 \right.\\
\hphantom{0=}{} - \frac{8}{15} c_2 c_4 \lambda_1 - \frac{11\lambda_1}{15} c_1^2 c_4 + \frac{2\lambda_1}{9} c_2^3 + \lambda_1 c_1^2 c_2^2 + \frac{4 \lambda_1}{9} c_1^4 c_2
 + \frac{4 \lambda_1}{3^2 \cdot 35} c_1^6 - \frac{2\lambda_1^2}{5} c_1 c_4 + \frac{3\lambda_1^2}{5} c_1 c_2^2 \\
\hphantom{0=}{} + \frac{53\lambda_1^2}{3^2 \cdot 35} c_1^3 c_2 + \frac{2 \lambda_1^2}{75} c_1^5 - \frac{2\cdot 853\lambda_1^3}{3^4 \cdot 7 \cdot 25} c_4 + \frac{14 \lambda_1^3}{81} c_2^2 + \frac{2 \cdot 71 \lambda _1^3}{3^4 \cdot 5} c_1^2 c_2 + \frac{79 \lambda_1^3}{3^5 \cdot 5} c_1^4 - \frac{2}{5}\lambda_4 c_1^3 \\
\hphantom{0=}{}
- \frac{13 \lambda_4}{15} c_1 c_2 + \frac{2 \cdot 13 \cdot 47 \lambda_1^4}{3^5 \cdot 25} c_1 c_2 + \frac{1249 \lambda_1^4}{3^6 \cdot 25} c_1^3
 - \frac{11 \lambda _4 \lambda _1}{45} c_1^2 - \frac{8 \lambda_4 \lambda_1}{21} c_2
 + \frac{2^2 \cdot 11^2 \cdot 13 \lambda_1^5}{3^6 \cdot 25 \cdot 7} c_2\\
\hphantom{0=}{}
 + \frac{2^8 \lambda _1^5}{3^5 \cdot 25} c_1^2
 + \frac{10 \lambda_6}{21} c_1
 - \frac{2^2 \cdot 11 \cdot 13 \lambda_4 \lambda_1^2}{3^3 \cdot 25 \cdot 7} c_1 + \frac{2^2 \cdot 17 \cdot 37 \cdot 43 \lambda _1^6}{3^8 \cdot 5^3 \cdot 7} c_1 + \frac{\lambda _7}{21}
 + \frac{5 \lambda_6 \lambda_1}{21}\\
\hphantom{0=}{}
 - \frac{2\cdot 167 \lambda_4 \lambda_1^3}{3^5 \cdot 35}
 + \frac{23011 \lambda_1^7}{3^8 \cdot 5^3 \cdot 7} + \left(c_2 + 2c_1^2 + \frac{5 \lambda_1}{3} c_1
 + \frac{3 \lambda_1^2}{5} \right) \wp_{1,4} - \left(\frac{c_2^2-c_4}{2} + 2 c_2 c_1^2\right.\\
\left.\hphantom{0=}{}
 + \frac{2}{3} c_1^4 + \frac{7 \lambda_1}{3} c_2 c_1 + \frac{10 \lambda_1}{9} c_1^3 + \frac{79 \lambda_1^2}{90} c_2
 + \frac{31 \lambda_1^2}{30} c_1^2 + \frac{509 \lambda _1^3}{3^4 \cdot 10} c_1
 - \frac{\lambda _4}{6} + \frac{2^5\cdot 11 \lambda_1^4}{3^4 \cdot 25} \right) \wp_{1,2}\\
\hphantom{0=}{}
 - \left(\big(c_2^2 - c_4\big) c_1 + \frac{4}{3} c_2 c_1^3 + \frac{4}{15} c_1^5 - \frac{11 \lambda _1}{30} c_4 + \frac{\lambda _1}{2} c_2^2 + \frac{4}{3}\lambda_1 c_2 c_1^2 - \frac{\lambda_4}{5} c_1
 + \frac{2 \lambda_1}{9} c_1^4\right.\\
\left.\hphantom{0=}{}
+ \frac{34 \lambda _1^2}{45} c_2 c_1+ \frac{11 \lambda_1^2}{45} c_1^3 + \frac{11\cdot 17\lambda _1^3}{3^4 \cdot 10} c_2 + \frac{7\cdot 29 \lambda _1^3}{3^4\cdot 10} c_1^2 + \frac{7^2\cdot 31 \lambda _1^4}{3^5 \cdot 50} c_1
 + \frac{61 \lambda_1^5}{3^4\cdot 25} - \frac{\lambda _1 \lambda _4}{18}
 \right) \wp_{1,1} \\
\hphantom{0=}{} + \frac{1}{15} \left(\wp_{1,1,1,4} - \frac{3}{2}\wp_{1,1,1}\wp_{1,1,2}
 + 6 \wp_{1,2}\wp_{1,1}^2 - 9 \wp_{1,1}\wp_{1,4} - \frac{3}{2} \wp_{1,2}\wp_{2,2}\right. \\
\left.\hphantom{0=}{}
+ \frac{3}{2} \lambda_1 \big(\wp_{1,2}^2 - \wp_{1,1}\wp_{2,2}\big) +\lambda_1 \wp_{2,4} - \lambda_1^2 \wp_{1,4}
 + \frac{3}{2} \lambda_1^2 \wp_{1,1}\wp_{1,2} - 2 \lambda_7\right)\\
\hphantom{0=}{} + \frac{1}{90} (6c_1 + \lambda_1) \big(4\wp_{1,1}^2 - \wp_{1,1,1}^2 + \wp_{1,2}^2
 - 4 \wp_{1,1} \wp_{2,2} + 4 \wp_{2,4} + 2 \lambda_1 \wp_{1,1} \wp_{1,2} + \lambda_1^2 \wp_{1,1}^2 \\
\hphantom{0=}{} + 4 \lambda_4 \wp_{1,1} - 4\lambda_6 \big) +
 \frac{1}{60} \big(\wp_{1,1,1,4} - \wp_{1,2,2,2} + 2 \wp_{1,1} \wp_{1,1,1,2} - 2 \wp_{1,2} \wp_{1,1,1,1} + \lambda_1 \wp_{1,1,2,2}\big)\\
\hphantom{0=}{}
 + \frac{7}{90} (3 c_1 + 2\lambda_1) \big(4 \wp_{1,2}^2 + 2 \wp_{1,1} \wp_{2,2} - \wp_{1,1,2,2} + \wp_{2,4}
 - \lambda_1 \wp_{1,4} \big) \\
\hphantom{0=}{} + \left(\frac{c_1}{6} + \frac{c_1^2}{3} + \frac{7 \lambda_1}{30} c_1 + \frac{8 \lambda_1^2}{135} - \frac{1}{30} \wp_{1,1} \right)
 \big(6 \wp_{1,1}\wp_{1,2} - \wp_{1,1,1,2} - 3 \wp_{1,4} - \lambda_1 \wp_{2,2} + \lambda_1^2 \wp_{1,2} \big)\\
\hphantom{0=}{}
 - \frac{1}{5!} \left(20 c_2 c_1 + \frac{40}{3} c_1^3 + \frac{20 \lambda_1}{3} \big(c_2 + c_1^2\big) + \frac{8 \lambda_1^2}{3} c_1
 + \frac{88}{81} \lambda_1^3 - (12 c_1 + 2\lambda_1)\wp_{1,1} - 14 \wp_{1,2} \right.\\
\left. \hphantom{0=}{} + \partial_{u_1}\partial_{u_2}+ \frac{1}{9}(6c_1 +\lambda_1)
 \partial_{u_1}^2 \right) \big(6 \wp_{1,1}^2 - \wp_{1,1,1,1} - 3\wp_{2,2} + 2 \lambda_1 \wp_{1,2} \\
\hphantom{0=}{} + \lambda_1^2 \wp_{1,1} + 2\lambda_4\big) \bigg)\xi^{-1} + O(1).
\end{gather*}

\section[Regularization of second kind integral on $(3,7)$-curve]{Regularization of second kind integral on $\boldsymbol{(3,7)}$-curve}\label{A:Reg37}
Given $(3,7)$-curve the formula \eqref{ZetaDef} acquires the form
\begin{gather*}
 \sum_{i=1}^{g=6} \int_{(z_i,w_i)}^{(x_i,y_i)} \rmd r = \begin{pmatrix}
 -\zeta_1(u) \\
 -\zeta_2(u) + \wp_{1,1}(u)\\
 -\zeta_4(u) + \mathcal{P}_4(u) \\
 -\zeta_5(u) + \mathcal{P}_5(u) \\
 -\zeta_8(u) + \mathcal{P}_8(u)\\
 -\zeta_{11}(u) + \mathcal{P}_{11}(u) \end{pmatrix}
 - \begin{pmatrix}
 -\zeta_1(v) \\
 -\zeta_2(v) + \wp_{1,1}(v)\\
 -\zeta_4(v) + \mathcal{P}_4(v) \\
 -\zeta_5(v) + \mathcal{P}_5(v) \\
 -\zeta_8(v) + \mathcal{P}_8(v)\\
 -\zeta_{11}(v) + \mathcal{P}_{11}(v) \end{pmatrix},
\end{gather*}
where in notation \eqref{wpNot}
\begin{gather*}
\mathcal{P}_4(u) = - \wp _{1,1,2} - \wp_{1,1}^2 + \lambda_2 \wp_{1,1}, \\
\mathcal{P}_5(u) = - \tfrac{1}{2} \wp _{1,2,2} - \tfrac{1}{6} \big(3 \wp_{1,1} - \lambda_2\big)\wp_{1,1,1} - \tfrac{5}{2}\wp_{1,1} \wp_{1,2} + \tfrac{5}{2} \wp_{1,4} -\tfrac{5}{12} \lambda_2, \\
\mathcal{P}_8(u) = -\wp_{2,2,4} + \tfrac{1}{3} \lambda_2 \wp_{1,1,4} +\big(\wp_{2,2} + \tfrac{2}{3} \lambda_2 \wp_{1,1}\big)\wp _{1,1,2}\\
\hphantom{\mathcal{P}_8(u) =}{} + 2\big(\wp_{1,1}\wp_{1,2} - \wp_{1,4} - \tfrac{1}{3} \lambda_2 \wp_{1,2} + \tfrac{1}{3} \lambda_5\big) \wp _{1,1,1}
 + \wp_{1,1}^4 + 4 \wp_{1,2}^2 \wp_{1,1} - \tfrac{4}{3} \lambda_2 \wp_{1,1}^3 \\
\hphantom{\mathcal{P}_8(u) =}{} - 4 \wp_{1,2} \wp_{1,4}- 4 \wp_{1,1} \wp_{1,5} + \tfrac{1}{3} \lambda_2^2 \wp_{1,1}^2+ 2\lambda_6 \wp_{1,1} - \tfrac{13}{21} \lambda_8 + \tfrac{3}{14} \lambda_2 \lambda_6, \\
\mathcal{P}_{11}(u) = - \tfrac{3}{8} \wp_{1,1,2} \wp_{1,1,1,4} - \tfrac{1}{4} \wp_{1,2} \wp_{2,2,4}
+ \big(\tfrac{5}{4} \wp_{2,2} + \lambda_2 \wp_{1,1} - \tfrac{4}{7}\lambda_2^2\big) \wp_{1,2,4}\\
\hphantom{\mathcal{P}_{11}(u) =}{} + \big(\tfrac{7}{2} \wp_{1,1} \wp_{1,2}- \tfrac{7}{4} \wp_{1,4} - \lambda_2 \wp_{1,2} + \tfrac{13}{12}\lambda_5 \big)\wp_{1,1,4} + \big( \tfrac{1}{8} \wp_{1,2}^2 - \tfrac{5}{4} \wp_{1,1} \wp_{2,2} - \lambda_2\wp_{1,1}^2 \\
\hphantom{\mathcal{P}_{11}(u) =}{}-\tfrac{1}{2} \wp_{1,5}+ \tfrac{3}{4} \wp_{2,4} + \tfrac{3}{7}\lambda_2^2\wp_{1,1} \big) \wp_{1,2,2} + \big({-} \tfrac{5}{4} \wp_{1,1}^2\wp_{1,2} + \tfrac{23}{8} \wp_{1,4}\wp_{1,1}
 - \tfrac{13}{16} \wp_{1,2}\wp_{2,2} \\
\hphantom{\mathcal{P}_{11}(u) =}{}
 - \tfrac{3}{8} \lambda_2 \wp_{1,1}\wp_{1,2}
 + \tfrac{5}{8} \wp_{2,5} - \tfrac{1}{2} \lambda_2 \wp_{1,4} + \tfrac{1}{7} \lambda_2^2\wp_{1,2} + \tfrac{1}{24} \lambda_5 \wp_{1,1}\big) \wp_{1,1,2} \\
\hphantom{\mathcal{P}_{11}(u) =}{}
+ \big({-} \tfrac{1}{2} \wp_{1,1}^4 - \tfrac{1}{4} \wp_{1,1}^2 \wp_{2,2} - \tfrac{7}{2} \wp_{1,1} \wp_{1,2}^2
 + \tfrac{2}{3} \lambda_2 \wp_{1,1}^3 + \tfrac{1}{4} \wp_{1,1} \big(10\wp_{1,5} + \wp_{2,4}\big)\\
\hphantom{\mathcal{P}_{11}(u) =}{}+ \tfrac{17}{8} \wp_{1,2} \wp_{1,4}+ \tfrac{7}{16} \wp_{2,2}^2 + \tfrac{7}{12} \lambda_2 \wp_{1,1}\wp_{2,2}
 + \tfrac{19}{24} \lambda_2 \wp_{1,2}^2 - \tfrac{1}{42} \lambda_2^2 \wp_{1,1}^2\\
\hphantom{\mathcal{P}_{11}(u) =}{} -\tfrac{1}{12} \lambda_2 (11\wp_{1,5} + \wp_{2,4}) - \tfrac{1}{7} \lambda_2^2 \wp_{2,2} -\tfrac{19}{24} \lambda_5 \wp_{1,2} -\tfrac{5}{4} \lambda_6 \wp_{1,1}
 - \tfrac{1}{7} \lambda_2^3 \wp_{1,1} +\tfrac{11}{12} \lambda_8 \big)\wp_{1,1,1}\\
\hphantom{\mathcal{P}_{11}(u) =}{} - \tfrac{11}{2} \big(\wp_{1,1}^4 \wp_{1,2}- \wp_{1,1}^3 \big(\wp_{1,4} + \lambda_2 \wp_{1,2}\big) + \wp_{1,2}^3 \wp_{1,1}- \wp_{1,2}^2 \wp_{1,4}\\
\hphantom{\mathcal{P}_{11}(u) =}{}- \wp_{1,1} \big(2 \wp_{1,2} \wp_{1,5} - \wp_{1,8}\big)
 + \lambda_2 \wp_{1,1}^2 \wp_{1,4} + \tfrac{1}{3} \lambda_5 \wp_{1,1}^3 + \wp_{1,4} \wp_{1,5} \big)\\
 \hphantom{\mathcal{P}_{11}(u) =}{}
 - \big(\tfrac{4}{7} \lambda_2^3 + \tfrac{7}{2} \lambda_6\big) \wp_{1,1} \wp_{1,2} + \tfrac{5}{6} \lambda_2 \lambda_5 \wp_{1,1}^2 + \big(\tfrac{4}{7} \lambda_2^3 + \tfrac{7}{2} \lambda_6\big) \wp_{1,4}\\
 \hphantom{\mathcal{P}_{11}(u) =}{}
 + \big(\tfrac{3}{7} \lambda_2^2 \lambda_5 + \tfrac{5}{2} \big) \lambda_9 \wp_{1,1}
 - \tfrac{2^7}{7!} \lambda_2^3 \lambda_5 + \tfrac{3}{20} \lambda_5 \lambda_6 + \tfrac{2^5\cdot 33}{7!} \lambda_2 \lambda_9 - \tfrac{33}{35} \lambda_{11}.
\end{gather*}
Here and in what follows the argument of Abelian functions is omitted for the sake of brevity.

\begin{Lemma}\label{L:ind4WPC37}
The following relations between Abelian functions on Jacobian of $\mathcal{V}_{(3,7)}$ hold
\begin{gather*}
 \wp_{1,1,1,1} = 6 \wp_{1,1}^2 - 3\wp_{2,2} - 4 \lambda_2 \wp_{1,1}, \\
 \wp_{1,1,1,2} = 6\wp_{1,1}\wp_{1,2} - 3 \wp_{1,4} - \lambda_2 \wp_{1,2} + \lambda_5, \\
 \wp_{1,1,2,2} = 4\wp_{1,2}^2 + 2\wp_{1,1}\wp_{2,2} - 4\wp_{1,5} + \wp_{2,4} - \lambda_2 \wp_{2,2} + 2\lambda_6, \\
 \wp_{1,1,2,4} = 4\wp_{1,2}\wp_{1,4} + 2\wp_{1,1}\wp_{2,4} + \wp_{4,4} - \lambda_2 \wp_{2,4}, \\
 \wp_{1,2,2,4} - \tfrac{1}{3}\lambda_2\wp_{1,1,1,4} = 4 \wp_{1,2} \wp_{2,4} + 2 \wp_{2,2} \wp_{1,4} - 2 \lambda_2 \wp_{1,1} \wp_{1,4} - 3\wp_{1,8} + 2\wp_{4,5}\\
 \qquad {} + \tfrac{4}{3} \lambda_2^2 \wp_{1,4} - 2\lambda_9, \\
 \wp_{2,2,2} = \wp_{1,1,4} + 2 \wp_{1,1}\wp_{1,1,2} - 2 \wp_{1,2}\wp_{1,1,1} - \lambda_2 \wp_{1,1,2}, \\
 \wp_{1,1,1}^2 = 4\wp_{1,1}^3 - 4\wp_{1,1}\wp_{2,2} + \wp_{1,2}^2 - 4\lambda_2 \wp_{1,1}^2 - 4\wp_{1,5} + 4\wp_{2,4}, \\
 \wp_{1,1,1}\wp_{1,1,2} - \tfrac{2}{3}\wp_{1,1,1,4} = 4 \wp_{1,1}^2 \wp_{1,2} - 6 \wp_{1,1}\wp_{1,4} - \wp_{1,2}\wp_{2,2} - \lambda_2 \wp_{1,1}\wp_{1,2} + 2\wp_{2,5} \\
 \qquad{} + \tfrac{8}{3} \lambda_2 \wp_{1,4} + 2\lambda_5 \wp_{1,1}, \\
 \wp_{1,1,1}\wp_{1,2,2} + \tfrac{1}{2} \wp_{1,1,2}^2 = 4 \wp_{1,1} \wp_{1,2}^2 + 2 \wp_{1,1}^2 \wp_{2,2} -4\wp_{1,1} \wp_{1,5} + 2\wp_{1,1} \wp_{2,4} - \wp_{1,2} \wp_{1,4} - \tfrac{3}{2} \wp_{2,2}^2 \\
 \qquad{} - \lambda_2 \big(2 \wp_{1,1} \wp_{2,2} + \wp_{1,2}^2\big) + 2\wp_{4,4} + 2\lambda_2 \big(\wp_{1,5} - \wp_{2,4}\big) + \lambda_5 \wp_{1,2} + 2\lambda_6 \wp_{1,1} - 2\lambda_8, \\
 \wp_{1,1,2}\wp_{1,2,2} = 2\wp_{1,1}\wp_{1,2}\wp_{2,2} + 2 \wp_{1,2}^3 - 4\wp_{1,2} \wp_{1,5} + 2 \wp_{1,2} \wp_{2,4} - \wp_{2,2} \wp_{1,4} - \lambda_2 \wp_{1,2}\wp_{2,2} \\
 \qquad {} + 2 \wp_{4,5} + \lambda_5 \wp_{2,2} + 2\lambda_6 \wp_{1,2} - 4\lambda_9, \\
 \wp_{1,1,1}\wp_{1,1,4} - \tfrac{2}{3} \wp_{1,1} \wp_{1,1,1,4} = \wp _{1,2} \wp _{2,4} - 2\wp_{2,2} \wp_{1,4} - \tfrac{4}{3} \lambda_2 \wp _{1,1} \wp_{1,4} - 2 \wp_{8,1}, \\
 \wp_{1,1,1}\wp_{1,2,4} + \wp_{1,1,2}\wp_{1,1,4} - \tfrac{1}{3} \wp_{1,2} \wp_{1,1,1,4}
 = 2\wp_{1,1}^2 \wp_{2,4} + 4 \wp_{1,1} \wp_{1,2} \wp_{1,4} + 2\wp_{1,1} \wp_{4,4} \\
 \qquad {} - \wp_{2,2} \wp_{2,4} - 2 \wp_{1,4}^2 - \tfrac{2}{3} \lambda_2 \big(3\wp_{1,1} \wp_{2,4} + \wp_{1,2} \wp_{1,4}\big) + 2\wp_{8,2} + 2\lambda_5 \wp_{1,4},\\
 \wp_{1,1,2}\wp_{1,1,4} + \tfrac{1}{2} \wp_{1,2,2}^2 + \tfrac{1}{2} \big(\wp_{1,1}-\lambda_2\big) \wp_{1,1,2}^2 - \tfrac{4}{3} \wp_{1,2} \wp_{1,1,1,4}\\
 \qquad{} = 2\big(\wp_{1,1}-\lambda_2\big) \wp_{1,1} \wp_{1,2}^2 - 6 \wp_{1,1} \wp_{1,2} \wp_{1,4} + 2 \wp_{1,2}^2 \wp_{2,2} +\tfrac{1}{2} \wp_{1,1} \wp_{2,2}^2 + 2 \wp_{1,2} \wp_{2,5} \\
 \qquad {} - 2 \wp_{2,2} \wp_{1,5} + \wp_{2,2} \wp_{2,4} - \tfrac{3}{2} \wp_{1,4}^2 + \tfrac{13}{2} \lambda_2 \wp_{1,2} \wp_{1,4} - \tfrac{1}{2}\lambda_2 \wp_{2,2}^2
 + \tfrac{1}{2} \lambda_2^2 \wp_{1,2}^2 + 2\lambda_5 \wp_{1,1} \wp_{1,2}\\
 \qquad {} - 2 \wp_{2,8} + 2 \wp_{5,5} + \lambda_5 \wp_{1,4} + 2 \lambda_6 \wp_{2,2} - \lambda_2 \lambda_5 \wp_{1,2} + \tfrac{4}{3} \lambda_2 \lambda_8 + \tfrac{1}{2} \lambda_5^2, \\
 \wp_{1,2,2} \wp_{1,1,4} + \wp_{1,1,2}\wp_{1,2,4} + \wp_{1,1,1}\wp_{2,2,4} - \big(\wp_{2,2} + \tfrac{2}{3} \lambda_2 \wp_{1,1}\big) \wp_{1,1,1,4}\\
 \qquad{} = 2 \wp_{1,1} \big(3\wp_{1,2} \wp_{2,4} - \wp_{2,2} \wp_{1,4} \big) + 2\wp_{1,2}^2 \wp_{1,4} - 4 \lambda_2 \wp_{1,1}^2 \wp_{1,4} - 6 \wp_{1,1} \wp_{1,8} + 2 \wp_{1,1} \wp_{4,5}\\
 \qquad {} + 2\wp_{1,4} \wp_{1,5} + \wp_{1,2} \wp_{4,4} - 5 \wp_{1,4} \wp_{2,4} - 2\lambda_2 \big(\wp_{1,2} \wp_{2,4} - \wp_{2,2} \wp_{1,4} \big)
 +\tfrac{8}{3} \lambda_2^2 \wp_{1,1} \wp_{1,4}\\
 \qquad {} + 2\lambda_2 \wp_{1,8} + \lambda_5 \wp_{2,4} + 2\lambda_6 \wp_{1,4} - 2\lambda_9 \wp_{1,1}- 2 \lambda_{11}, \\
 \wp_{1,1,4}^2 - \tfrac{4}{3} \wp _{1,4} \wp _{1,1,1,4} =
 - 4 \wp _{1,1} \wp_{1,4}^2 + \wp_{2,4}^2 + \tfrac{4}{3}\lambda_2 \wp _{1,4}^2 - 4 \wp_{1,11}, \\
 \wp_{1,1,4}\wp_{1,2,4} - \tfrac{1}{3} \wp_{2,4} \wp _{1,1,1,4} = 2\wp_{1,2} \wp_{1,4}^2 +
 2\wp_{1,4} \wp_{4,4} - \tfrac{2}{3}\lambda_2 \wp_{1,4}\wp_{2,4} + 2\wp_{2,11}.
\end{gather*}
\end{Lemma}
The proof is similar to the proofs of Lemmas~\ref{L:ind4WPC34} and \ref{L:ind4WPC35}.

\section[Regularization of second kind integral on $(4,5)$-curve]{Regularization of second kind integral on $\boldsymbol{(4,5)}$-curve}\label{A:Reg45}
In the case of $(4,5)$-curve, the formula \eqref{ZetaDef} gets the form (notation \eqref{wpNot} is used)
\begin{gather*}
 \sum_{i=1}^{g=6} \int_{(z_i,w_i)}^{(x_i,y_i)} \rmd r = \begin{pmatrix}
 -\zeta_1(u) \\
 -\zeta_2(u) + \wp_{1,1}(u)\\
 -\zeta_3(u) + \mathcal{P}_3(u) \\
 -\zeta_6(u) + \mathcal{P}_6(u) \\
 -\zeta_7(u) + \mathcal{P}_7(u)\\
 -\zeta_{11}(u) + \mathcal{P}_{11}(u) \end{pmatrix}
 - \begin{pmatrix}
 -\zeta_1(v) \\
 -\zeta_2(v) + \wp_{1,1}(v)\\
 -\zeta_3(v) + \mathcal{P}_3(v) \\
 -\zeta_4(v) + \mathcal{P}_6(v) \\
 -\zeta_7(v) + \mathcal{P}_7(v)\\
 -\zeta_{11}(v) + \mathcal{P}_{11}(v) \end{pmatrix},
\end{gather*}
where
\begin{gather*}
\mathcal{P}_3(u) = \tfrac{1}{2} \wp _{1,1,1} + \tfrac{3}{2} \wp_{1,2}, \\
\mathcal{P}_6(u) = - \wp_{1,2,3} - \big(\wp _{1,1}- \tfrac{1}{2} \lambda_2 \big) \wp_{1,1,2} - \tfrac{1}{2} \big(\wp _{1,2}- \lambda_3\big) \wp_{1,1,1} - \wp_{1,1}^3 - 3 \wp_{1,1} \wp_{1,3} - \tfrac{3}{2} \wp_{1,2}^2\\
 \hphantom{\mathcal{P}_6(u) =}{} + \tfrac{3}{2}\lambda_2 \wp_{1,1}^2 + \tfrac{3}{2} \lambda_3 \wp_{1,2} - \tfrac{1}{2} \lambda_2^2 \wp_{1,1} - \tfrac{7}{10}\lambda_6, \\
\mathcal{P}_7(u) = -\tfrac{1}{2}\wp_{1,3,3} - \big(\wp_{1,1}- \tfrac{3}{5}\lambda_2\big) \big(\wp _{1,1,3} + \tfrac{1}{2}\wp _{1,2,2} \big)- \big(\wp_{1,2} - \tfrac{1}{4} \lambda_3\big)\wp _{1,1,2}- \tfrac{1}{2}\big(\wp_{1,1}^2 + \wp_{1,3} \\
\hphantom{\mathcal{P}_7(u) =}{} - \tfrac{8}{5}\lambda_2\wp_{1,1} + \tfrac{3}{5}\lambda_2^2\big) \wp _{1,1,1}- \tfrac{7}{2} \big(\wp_{1,1}^2 \big(\wp_{1,2} - \tfrac{1}{2} \lambda_3\big) + \wp_{1,2} \big(\wp_{1,3}-\lambda_2\wp_{1,1}\big) - \wp_{1,6} \big) \\
\hphantom{\mathcal{P}_7(u) =}{}- \tfrac{3}{5} \lambda_2^2 \wp_{1,2}- \tfrac{23}{20} \lambda_2 \lambda_3 \wp_{1,1} - \tfrac{2\cdot 7^2}{5!} \lambda_7,\\
\mathcal{P}_{11}(u) = \tfrac{3}{2} \wp_{1,1,1} \wp _{3,3,1,1} + \big(\tfrac{1}{3} \wp_{1,2,2} - \tfrac{1}{6}\wp_{1,1,3} - \big(\tfrac{3}{4}\wp_{1,1} + \tfrac{5}{9}\lambda_2\big) \wp_{1,1,1} \big) \wp _{3,1,1,1}\\
\hphantom{\mathcal{P}_{11}(u) =}{} + \big(\tfrac{5}{2}\wp_{1,1}^2 + \tfrac{1}{2} \wp_{3,1} - \tfrac{1}{2} \wp_{2,2} - \tfrac{5^2}{18}\lambda_2 \wp_{1,1}\big) \wp_{3,3,1}
 + \big(\tfrac{5}{2} \wp_{1,1} \wp_{1,2} + \tfrac{1}{2} \wp_{3,2} - \tfrac{1}{3} \lambda_2 \wp_{2,1} \big) \wp_{3,2,1}\\
\hphantom{\mathcal{P}_{11}(u) =}{}
 + \big(4\wp_{1,1}^3 + 4 \wp_{1,1} \wp_{1,3} - \tfrac{3}{4} \wp_{1,2}^2 - 2 \wp_{1,1} \wp_{2,2} - \tfrac{17}{3} \lambda_2 \wp_{1,1}^2 + \tfrac{2}{3} \wp_{3,3}\\
\hphantom{\mathcal{P}_{11}(u) =}{}
 - \tfrac{1}{9}\lambda_2 \big(11 \wp_{1,3} - 7 \wp_{2,2}\big) + \lambda_3 \wp_{1,2}
 + \tfrac{5}{3} \lambda_2^2 \wp_{1,1} - \tfrac{3}{10} \lambda_3^2 - \tfrac{1}{5} \lambda_6\big) \wp_{1,1,3} \\
\hphantom{\mathcal{P}_{11}(u) =}{}
 - \big(\wp_{1,1}^3 + \wp_{1,1} \big(\tfrac{3}{2} \wp_{1,3} - \wp_{2,2}\big) - \tfrac{7}{8} \wp_{1,2}^2
 - \tfrac{23}{18} \lambda_2 \wp_{1,1}^2 - \tfrac{1}{6} \wp_{3,3}
 - \tfrac{1}{9} \lambda_2 \big(\wp_{1,3} - \tfrac{5}{2} \wp_{2,2}\big) \\
\hphantom{\mathcal{P}_{11}(u) =}{} + \tfrac{1}{4} \lambda_3 \wp_{1,2} + \tfrac{5}{18} \lambda_2^2 \wp_{1,1} + \tfrac{3}{20} \lambda_3^2 - \tfrac{2}{5} \lambda_6\big) \wp_{1,2,2}
 + \big(\tfrac{1}{2} \wp_{1,1}^2 \wp_{1,2} + \tfrac{3}{2} \wp_{1,1} \wp_{2,3} \\
\hphantom{\mathcal{P}_{11}(u) =}{} + \tfrac{5}{4} \wp_{1,2} \big(2 \wp_{1,3} - 5 \wp_{2,2} \big) - \tfrac{7}{3} \lambda_2 \wp_{1,1} \wp_{1,2} + \tfrac{1}{2} \lambda_3 \wp_{1,1}^2 - \tfrac{2}{9} \lambda_2 \wp_{2,3}
 + \tfrac{1}{4} \lambda_3 \big(\wp_{2,2} - 4\wp_{1,3}\big) \\
\hphantom{\mathcal{P}_{11}(u) =}{} + \tfrac{5}{9} \lambda_2^2 \wp_{1,2} + \tfrac{7}{9} \lambda_2 \lambda_3 \wp_{1,1} + \tfrac{3}{4} \lambda_7 \big) \wp_{1,1,2} + \big(\tfrac{1}{2} \wp_{1,1}^4 + \tfrac{3}{4} \wp_{1,1}^2 \big(\wp_{2,2} + 4\wp_{1,3}\big) + \tfrac{19}{4} \wp_{1,1} \wp_{1,2}^2\\
\hphantom{\mathcal{P}_{11}(u) =}{} - \tfrac{23}{18} \lambda_2 \wp_{1,1}^3 - 4 \wp_{1,1} \wp_{3,3} - 3 \wp_{1,2} \wp_{2,3} + \tfrac{3}{4} \wp_{2,2} \wp_{1,3} - 6 \wp_{1,3}^2 + \tfrac{1}{2} \wp_{2,2}^2\\
\hphantom{\mathcal{P}_{11}(u) =}{}
 + \tfrac{1}{18} \lambda_2 \big(79 \wp_{1,3} - \tfrac{37}{2} \wp_{2,2}\big) \wp_{1,1}
 - \tfrac{23}{24} \lambda_2 \wp_{1,2}^2
 - \tfrac{37}{8} \lambda_3 \wp_{1,1} \wp_{1,2} + \tfrac{19}{18} \lambda_2^2 \wp_{1,1}^2 \\
\hphantom{\mathcal{P}_{11}(u) =}{}
 + \tfrac{5}{4} \big(\wp_{2,6} + 2 \wp_{1,7} \big) + \tfrac{5^2}{18} \lambda_2 \wp_{3,3} + \tfrac{3}{2} \lambda_3 \wp_{2,3} + \tfrac{5}{18} \lambda_2^2 \big(\wp_{2,2} - 4 \wp_{1,3}\big)
 + \tfrac{5}{6} \lambda_2 \lambda_3 \wp_{1,2} \\
\hphantom{\mathcal{P}_{11}(u) =}{}
 + \big(\tfrac{7^2}{10} \lambda_6 - \tfrac{3}{20} \lambda_3^2 - \tfrac{5}{18} \lambda_2^3\big) \wp_{1,1}
 + \tfrac{3}{20} \lambda_2 \lambda_3^2 + \tfrac{37}{30} \lambda_2 \lambda_6 - \tfrac{5}{2} \lambda_8\big) \wp_{1,1,1}\\
\hphantom{\mathcal{P}_{11}(u) =}{}
 + \tfrac{11}{2} \Big(\wp_{1,1}^4 \big(\wp_{1,2} - \tfrac{1}{2}\lambda_3\big) + \wp_{1,2}^3 \big(\wp_{1,1} - \tfrac{1}{3} \lambda_2) + \wp_{1,1}^2 \wp_{1,2} \big(3\wp_{1,3} - 2 \lambda_2 \wp_{1,1} \big) \\
\hphantom{\mathcal{P}_{11}(u) =}{}
 - \wp_{1,1}^2 \big(\wp_{1,6} + \lambda_3 \wp_{1,3} \big) + \wp_{1,2} \wp_{1,3}^2
 - 2 \lambda_2 \wp_{1,1} \wp_{1,2} \wp_{1,3} - \lambda_3 \wp_{1,1} \wp_{1,2}^2 - \wp_{1,2} \wp_{1,7}\\
\hphantom{\mathcal{P}_{11}(u) =}{}
- \big(\wp_{1,3} - \lambda_2 \wp_{1,1} \big) \wp_{1,6} + \lambda_6 \wp_{1,1} \wp_{1,2}
 + \tfrac{1}{2} \lambda_7 \wp_{1,1}^2\Big) + \tfrac{7\cdot 17}{18} \lambda_2^2 \wp_{1,1}^2 \wp_{1,2}
 + \tfrac{79}{18} \lambda_2\lambda_3 \wp_{1,1}^3 \\
\hphantom{\mathcal{P}_{11}(u) =}{}
 + \tfrac{10}{9} \lambda_2^2 \wp_{1,2} \wp_{1,3}
 + \tfrac{13}{12} \lambda_2 \lambda_3 \big(2 \wp_{1,1} \wp_{1,3} + \wp_{1,2}^2\big)
 - \tfrac{10}{9} \lambda_2^3 \wp_{1,1} \wp_{1,2} - \tfrac{59}{6^2} \lambda_2^2 \lambda_3 \wp_{1,1}^2\\
\hphantom{\mathcal{P}_{11}(u) =}{}
 - \tfrac{10}{9} \lambda_2^2 \wp_{1,6} + \big(\tfrac{3}{10} \lambda_2\lambda_3^2 - \tfrac{7^2}{30} \lambda_2 \lambda_6 + \tfrac{5}{2} \lambda_8 \big) \wp_{1,2}
 + \big(\tfrac{9}{20} \lambda_3^3 - \tfrac{17}{10} \lambda_3 \lambda_6 - \tfrac{59}{6^2} \lambda_2 \lambda_7\big) \wp_{1,1} \\
\hphantom{\mathcal{P}_{11}(u) =}{}
 - \tfrac{2\cdot 211}{7!} \lambda_2 \lambda_3 \lambda_6 + \tfrac{2\cdot 277}{3\cdot 7!} \lambda_2^2 \lambda_7 + \tfrac{22}{45} \lambda_3 \lambda_8 - \tfrac{11^2}{90}\lambda_{11}.
\end{gather*}
The argument of Abelian functions is omitted for brevity.

\begin{Lemma}\label{L:ind4WPC45}
The following relations between Abelian functions on Jacobian of $\mathcal{V}_{(4,5)}$ hold
\begin{gather*}
 \wp_{1,1,1,1} = 6 \wp_{1,1}^2 - 3\wp_{2,2} + 4\wp_{1,3} - 4 \lambda_2 \wp_{1,1}, \\
 \wp_{1,1,1,2} = 6\wp_{1,1}\wp_{1,2} - 2 \wp_{2,3} - \lambda_2 \wp_{1,2} - 3 \lambda_3 \wp_{1,1}, \\
 \wp_{1,1,2,2} + \tfrac{2}{3} \wp_{1,1,1,3} = 4\wp_{1,1}\wp_{1,3} + 2 \wp_{1,1} \wp_{2,2} + 4 \wp_{1,2}^2 - \tfrac{4}{3} \wp_{3,3} + \tfrac{4}{3} \lambda_2 \wp_{1,3} - \lambda_2 \wp_{2,2}\\
 \hphantom{\wp_{1,1,2,2} + \tfrac{2}{3} \wp_{1,1,1,3} =}{} - 3\lambda_3 \wp_{1,2} + 2\lambda_6, \\
 \wp_{1,1,2,3} = 4\wp_{1,2}\wp_{1,3} + 2\wp_{1,1}\wp_{2,3} - 2 \wp_{1,6} - \lambda_2 \wp_{2,3} - \lambda_3 \wp_{1,3} + \lambda_7, \\
 \wp_{2,2,2} = 2\wp_{1,2,3} + (2 \wp_{1,1}- \lambda_2) \wp_{1,1,2} - (2 \wp_{1,2} - \lambda_3 ) \wp_{1,1,1}, \\
 \wp_{2,2,3} = \wp_{1,3,3} + (\wp_{1,1} - \lambda_2)\wp_{1,1,3} + \wp_{1,1} \wp_{1,2,2} -\tfrac{1}{2} \wp_{1,2}\wp_{1,1,2} -\tfrac{1}{2} (\wp_{2,2} + 2\wp_{1,3}) \wp_{1,1,1}, \\
 \wp_{2,3,3} = \wp_{1,1} \wp_{1,2,3} - \tfrac{1}{2} \wp_{1,2} \wp_{1,2,2} + (2\wp_{1,2}-\lambda_3) \wp_{1,1,3} - (\wp_{1,1}^2 + 3\wp_{1,3} - \wp_{2,2} - \lambda_2 \wp_{1,1})\wp_{1,1,2} \\
 \hphantom{\wp_{2,3,3} =}{} +\tfrac{1}{2} (2\wp_{1,1}\wp_{1,2} - 2\wp_{2,3} -\lambda_2 \wp_{1,2} - \lambda_3 \wp_{1,1})\wp_{1,1,1},\\
 \wp_{3,3,3} = -(\wp_{1,1} - \lambda_2)\wp_{1,3,3} +\tfrac{1}{2} (3\wp_{1,2}-2\lambda_3) \wp_{1,2,3} + \tfrac{1}{2} (4\wp_{1,1}^2-2\wp_{1,3} -\wp_{2,2}-2\lambda_2\wp_{1,1})\wp_{1,2,2} \\
 \hphantom{\wp_{3,3,3} =}{} - (3\wp_{1,1}^2+\wp_{1,3}-2\wp_{2,2}-2\lambda_2 \wp_{1,1})\wp_{1,1,3} - \tfrac{1}{2}(2\wp_{2,3}+\lambda_3 \wp_{1,1}) \wp_{1,1,2} + \tfrac{1}{6}(5\wp_{1,1,1,3} \\
\hphantom{\wp_{3,3,3} =}{} - 6\wp_{1,2}^2 - 6\wp_{1,1} \wp_{2,2} - 12 \wp_{1,1} \wp_{1,3} + 4\wp_{3,3} - 4\lambda_2 \wp_{1,3} + 3\lambda_2 \wp_{2,2} + 6\lambda_3 \wp_{1,2}\\
\hphantom{\wp_{3,3,3} =}{} - 6\lambda_6)\wp_{1,1,1},\\
 \wp_{1,1,1}^2 + \tfrac{4}{3} \wp_{1,1,1,3} = 4\wp_{1,1}^3 + 12 \wp_{1,1} \wp_{1,3}
 - 4 \wp_{1,1} \wp_{2,2} + \wp_{1,2}^2 + \tfrac{4}{3} \wp_{3,3} - 4\lambda_2 \wp_{1,1}^2 - \tfrac{4}{3}\lambda_2 \wp_{1,3},\\
 \wp_{1,1,1}\wp_{1,1,2} = 4 \wp_{1,1}^2 \wp_{1,2} - 2 \wp_{1,1} \wp_{2,3} - \wp_{1,2} \wp_{2,2} + 2 \wp_{1,2} \wp_{1,3}
 - 2\wp_{1,6} - 2 \lambda_2 \wp_{1,1} \wp_{1,2} - 2\lambda_3 \wp_{1,1}^2, \\
 \wp_{1,1,2}^2 - 2 \wp_{1,1,3,3} + \tfrac{4}{3} \lambda_2 \wp_{1,1,1,3} = 4 \wp_{1,1} \wp_{1,2}^2
 + \wp_{2,2}^2 - 8 \wp_{1,3}^2 - 4 \wp_{1,2} \wp_{2,3} - 4\wp_{1,1} \wp_{3,3} + 2\wp_{2,6} \\
 \qquad{} + 4\wp_{1,7} + 8 \lambda_2 \wp_{1,1} \wp_{1,3} + \tfrac{4}{3} \lambda_2 \wp_{3,3} - 4 \lambda_3 \wp_{1,1} \wp_{1,2}
 + 2\lambda_3 \wp_{2,3} - \tfrac{4}{3} \lambda_2^2 \wp_{1,3} + 4\lambda_6 \wp_{1,1} - 4\lambda_8,\\
 \wp_{1,1,1}\wp_{1,2,2} + 2 \wp_{1,1,3,3} + \tfrac{2}{3} (\wp_{1,1} - 2\lambda_2) \wp_{1,1,1,3} =
 2\wp_{1,1}\wp_{1,2}^2 + 2\wp_{1,1}^2 \wp_{2,2} + 4 \wp_{1,1}^2 \wp_{1,3} + 8\wp_{1,3}^2 \\
 \qquad{} + \tfrac{8}{3} \wp_{1,1} \wp_{3,3} + 2 \wp_{1,2} \wp_{2,3}
 - 2\wp_{2,2}^2 + 2 \wp_{1,3} \wp_{2,2} - \tfrac{1}{3} \lambda_2 \big(20 \wp_{1,1} \wp_{1,3}
 + 6 \wp_{1,1} \wp_{2,2} + 3 \wp_{1,2}^2 \big) \\
 \qquad{} - \lambda_3 \wp_{1,1} \wp_{1,2} - 8 \wp_{1,7} - \tfrac{4}{3} \lambda_2 \wp_{3,3} -2\lambda_3 \wp_{2,3}
 + \tfrac{4}{3} \lambda_2^2 \wp_{1,3} + 4 \lambda_8,\\
 \wp_{1,1,1}\wp_{1,1,3} + \wp_{1,1,3,3} - \tfrac{2}{3} \wp_{1,1} \wp_{1,1,1,3} =
 6 \wp_{1,3}^2 - 2 \wp_{1,3} \wp_{2,2} + \wp_{1,2} \wp_{2,3}
 + \tfrac{4}{3} \wp_{1,1} \wp_{3,3} + \wp_{2,6} \\
 \qquad{} - 4 \wp_{1,7} - \tfrac{4}{3} \lambda_2 \wp_{1,1}\wp_{1,3} - \lambda_3 \wp_{2,3}+2\lambda_8,\\
 \wp_{1,1,2}\wp_{1,2,2} + \wp_{1,1,1}\wp_{1,2,3} + \big(\wp_{1,2} - \tfrac{2}{3}\lambda_3\big) \wp_{1,1,1,3} =
 2\wp_{1,1}^2 \wp_{2,3} + 8 \wp_{1,1} \wp_{1,2} \wp_{1,3} \\
 \qquad{} + 2 \wp_{1,1} \wp_{1,2} \wp_{2,2} + 2 \wp_{1,2}^3 - 2 \wp_{1,1} \wp_{1,6} - 2 \wp_{2,2} \wp_{2,3} + 2 \wp_{1,3} \wp_{2,3}\\
 \qquad{} - \lambda_2 \big(2\wp_{1,1}\wp_{2,3} + \wp_{1,2} \wp_{2,2}\big)- \lambda_3 \big(4\wp_{1,1}\wp_{1,3} + \wp_{1,1} \wp_{2,2} + 2 \wp_{1,2}^2 \big) + 2\wp_{3,6} -\tfrac{2}{3} \lambda_3 \wp_{3,3}\\
 \qquad{}+\tfrac{2}{3}\lambda_2 \lambda_3 \wp_{1,3} + 2 \lambda_6 \wp_{1,2} + 2 \lambda_7 \wp_{1,1}, \\
 \wp_{1,1,2}\wp_{1,1,3} + \wp_{1,1,1}\wp_{1,2,3} - \tfrac{1}{3} \wp_{1,2} \wp_{1,1,1,3} =
 2 \wp_{1,1}^2 \wp_{2,3} + 4 \wp_{1,1} \wp_{1,2} \wp_{1,3} - 2\wp_{1,1} \wp_{1,6}\\
 \qquad{} + \tfrac{2}{3} \wp_{1,2} \wp_{3,3} - \wp_{2,2} \wp_{2,3} - \tfrac{2}{3} \lambda_2 \big(\wp_{1,2}\wp_{1,3} + 3 \wp_{1,1} \wp_{2,3}\big) - 2\lambda_3 \wp_{1,1} \wp_{1,3} + 2\wp_{2,7} + 2 \lambda_7 \wp_{1,1},\\
 \wp_{1,1,2}\wp_{1,2,3} + \tfrac{1}{2} \wp_{1,2,2}^2 - \tfrac{1}{2} \wp_{1,1,3}^2
 + (\wp_{1,1}-\lambda_2) \wp_{1,1,3,3} + \tfrac{1}{3} \big(2\wp_{1,3} + \wp_{2,2} - 2\lambda_2\wp_{1,1}\\
 \qquad{} + 2\lambda_2^2 \big) \wp_{1,1,1,3} = 2 \wp_{1,1}^2 \big(\wp_{3,3} - 2\lambda_2 \wp_{1,3} \big) + 2\wp_{1,2}^2 (\wp_{2,2}
 + \wp_{1,3}) + 6 \wp_{1,3}^2 \wp_{1,1}\\
 \qquad{}+ 2 \wp_{1,1} \wp_{2,2} \wp_{1,3} + 2 \wp_{1,1} \wp_{1,2} \wp_{2,3} - 4 \wp_{1,2} \wp_{1,6} - \tfrac{1}{2} \wp_{2,3}^2
 + \tfrac{2}{3} \big(\wp_{1,3} \wp_{3,3} - \wp_{2,2} \wp_{3,3}\big)\\
 \qquad{} - \wp_{1,1} \big(\wp_{2,6} + 2\wp_{1,7}\big) + \tfrac{2}{3} \lambda_2 \big(\wp_{1,3} \wp_{2,2} - 7 \wp_{1,3}^2 - 3 \wp_{1,2} \wp_{2,3} - 4 \wp_{1,1} \wp_{3,3}\big) \\
 \qquad{} - \lambda_3 \big(\wp_{1,1} \wp_{2,3} + 2\wp_{1,2} \wp_{2,2}\big)
 + \tfrac{1}{6} \lambda_2^2 \big(28 \wp_{1,1} \wp_{1,3} + 3 \wp_{1,2}^2 \big) - \lambda_2 \lambda_3 \wp_{1,1} \wp_{1,2}\\
\qquad{} + \tfrac{1}{2} \big( \lambda_3^2 - 4\lambda_6 \big) \wp_{1,1}^2 + 2\wp_{3,7}
 + \lambda_2 \big(2\wp_{1,7} + \wp_{2,6}\big) + 2\lambda_3 \wp_{1,6} + \tfrac{2}{3} \lambda_2^2 \wp_{3,3} + \lambda_2 \lambda_3 \wp_{2,3} \\
\qquad{} - \tfrac{2}{3} \big(3 \lambda_6 + \lambda_3^2\big) \wp_{1,3} + 2\lambda_6 \wp_{2,2}
 + 2\lambda_7 \wp_{1,2} + (\lambda_8 + 2\lambda_2 \lambda_6) \wp_{1,1} - 2 \lambda_2 \lambda_8 - 2 \lambda_{10},\\
 \wp_{1,2,2}\wp_{1,1,3} - \wp_{1,1,1}\wp_{1,3,3} - \tfrac{3}{2} \wp_{1,1,3}^2
 + (\wp_{1,1} - \lambda_2) \wp_{1,1,3,3} + \tfrac{2}{3} \big(3\wp_{1,3} - \wp_{2,2}\big) \wp_{1,1,1,3}\\
 \qquad{} = 8 \wp_{1,1} \wp_{1,3}^2 + 2 \wp_{1,2}^2 \wp_{1,3} - 2 \wp_{1,1} \wp_{2,2} \wp_{1,3}
 + \wp_{1,1} \wp_{2,6} - \wp_{1,2} \wp_{1,6} + \tfrac{4}{3} \wp_{2,2} \wp_{3,3} - 2 \wp_{1,3} \wp_{3,3}\\
 \qquad{} + \tfrac{1}{2} \wp_{2,3}^2 - \tfrac{1}{3} \lambda_2 \big(3 \wp_{1,2} \wp_{2,3}
 -2 \wp_{2,2} \wp_{1,3} + 12 \wp_{1,3}^2 \big) - \lambda_3 \wp_{1,2} \wp_{1,3} + 2 \wp_{3,7}\\
\qquad{} + \lambda_2 \big(2\wp_{1,7} - \wp_{2,6} \big) + \lambda_2 \lambda_3 \wp_{2,3} + \lambda_7 \wp_{1,2} - 2\lambda_2 \lambda_8,\\
 \wp_{1,2,2}\wp_{1,2,3} + \wp_{1,1,2}\wp_{1,3,3} - (\wp_{1,1} - \lambda_2) \wp_{1,1,1}\wp_{1,2,3}
 - \tfrac{1}{3} \big( \wp_{1,1} \wp_{1,2} - 2 \wp_{2,3} - \lambda_3 \wp_{1,1}\big) \wp_{1,1,1,3} \\
 \qquad{} = 2\wp_{1,1}^2 \big(\wp_{1,6} - 2 \wp_{1,2} \wp_{1,3} - \wp_{1,1} \wp_{2,3} + 2 \lambda_2 \wp_{2,3}
 + \lambda_3 \wp_{1,3}\big)\\
 \qquad{} + 2 \big(\wp_{1,2}\wp_{1,3} + \wp_{1,1} \wp_{2,3} \big) \big(\wp_{1,3} + \wp_{2,2} \big)
 + 2 \wp_{1,2}^2 \wp_{2,3} + \tfrac{2}{3} \wp_{1,1} \wp_{1,2} \big(\wp_{3,3} + 5\lambda_2 \wp_{1,3} \big)\\
 \qquad{} -2\wp_{1,1} \wp_{2,7} - 4 \wp_{1,2} \wp_{1,7} - \big(2\wp_{1,3} + \wp_{2,2}\big) \wp_{1,6} - \tfrac{4}{3} \wp_{2,3} \wp_{3,3} - 2\lambda_2 \wp_{1,1} \wp_{1,6}\\
 \qquad{} + \tfrac{2}{3} \lambda_2 \big(2\wp_{1,3} - 3\wp_{2,2}\big) \wp_{2,3}
 - \tfrac{1}{3} \lambda_3 \big(2 \wp_{1,1} \wp_{3,3} + 3 \wp_{2,2}\wp_{1,3}\ + 6 \wp_{1,2} \wp_{2,3}\big) - 2\lambda_2^2 \wp_{1,1} \wp_{2,3} \\
\qquad{} - \tfrac{4}{3} \lambda_2 \lambda_3 \wp_{1,1} \wp_{1,3}
 - 2 \lambda_7 \wp_{1,1}^2 + 2\lambda_2 \wp_{2,7} + 2\lambda_3 \wp_{1,7} + 2\lambda_6 \wp_{2,3}
 + \lambda_7 \wp_{2,2} + 2 \lambda_8 \wp_{1,2} \\
 \qquad{} + 2\lambda_2 \lambda_7 \wp_{1,1} - 2\lambda_{11}.
\end{gather*}
\end{Lemma}
The proof uses the idea of the proofs of Lemmas~\ref{L:ind4WPC34} and \ref{L:ind4WPC35}.

\vspace{-1mm}

\subsection*{Acknowledgements}
These results were presented and discussed on seminars of the School of Mathematics at Universities of Leeds, Loughborough, Glasgow, Edinburgh, and Heriot-Watt University. The authors are grateful to V.~Enolski for stimulating discussions, A.V.~Mikhailov, O.~Chalykh, A.P.~Veselov, C.~Athorne, H.W.~Braden, J.C.~Eilbeck, for hospitality and substantive comments. The authors are grateful to the anonymous referees for essential contribution to improvement of the paper.

\pdfbookmark[1]{References}{ref}
\LastPageEnding

\end{document}